\documentclass [11pt,reqno]{amsart}
\usepackage {amsmath,amssymb,verbatim,geometry,color}
\usepackage[pdftex,hyperindex]{hyperref}
\usepackage[pdftex]{graphicx}
\newcommand{\A}{\mathbf{A}}

\newcommand{\C}{\mathbf{C}}
\newcommand{\G}{\mathbf{G}}
\newcommand{\N}{\mathbf{N}}
\renewcommand{\P}{\mathbf{P}}
\newcommand{\Q}{\mathbf{Q}}
\newcommand{\R}{\mathbf{R}}
\newcommand{\Z}{\mathbf{Z}}

\newcommand{\fm}{\mathfrak{m}}

\newcommand{\bbA}{\mathbb{A}}
\newcommand{\bbG}{\mathbb{G}}
\newcommand{\bbP}{\mathbb{P}}

\newcommand{\cD}{\mathcal{D}}
\newcommand{\cE}{\mathcal{E}}
\newcommand{\cF}{\mathcal{F}}

\newcommand{\cH}{\mathcal{H}}

\newcommand{\cL}{\mathcal{L}}
\newcommand{\cM}{\mathcal{M}}
\newcommand{\cN}{\mathcal{N}}
\newcommand{\cO}{\mathcal{O}}

\newcommand{\cU}{\mathcal{U}}
\newcommand{\cV}{\mathcal{V}}

\newcommand{\cX}{\mathcal{X}}
\newcommand{\cY}{\mathcal{Y}}

\newcommand{\tcL}{\widetilde{\mathcal{L}}}
\newcommand{\tcX}{\widetilde{\mathcal{X}}}

\newcommand{\Lsch}{L^{\mathrm{sch}}}

\newcommand{\Xan}{X^{\mathrm{an}}}

\newcommand{\Xdiv}{X^{\mathrm{div}}}

\newcommand{\Xqm}{X^{\mathrm{qm}}}

\newcommand{\Xsch}{X^{\mathrm{sch}}}

\newcommand{\Xval}{X^{\mathrm{val}}}

\newcommand{\Yval}{Y^{\mathrm{val}}}

\newcommand{\cLan}{{\mathcal{L}^{\mathrm{an}}}}
\newcommand{\cUan}{{\mathcal{U}^{\mathrm{an}}}}
\newcommand{\cXan}{{\mathcal{X}^{\mathrm{an}}}}

\newcommand{\Dad}{D^{\mathrm{ad}}}
\newcommand{\Mad}{M^{\mathrm{ad}}}

\renewcommand{\a}{\alpha}

\renewcommand{\d}{\delta}
\newcommand{\e}{\varepsilon}
\newcommand{\f}{\varphi}
\newcommand{\unipar}{\varpi}

\newcommand{\la}{\lambda}

\newcommand{\p}{\psi}

\newcommand{\eg}{{\rm e.g.\ }} 
\newcommand{\ie}{{\rm i.e.\ }} 
\newcommand{\loccit}{\textit{loc.\,cit.}}

\DeclareMathOperator{\DF}{DF}

\DeclareMathOperator{\Ent}{Ent}

\DeclareMathOperator{\FS}{FS}
\DeclareMathOperator{\Fut}{Fut}

\DeclareMathOperator{\lct}{lct}

\DeclareMathOperator{\LM}{LM}
\DeclareMathOperator{\MA}{MA}

\DeclareMathOperator{\RLM}{RLM}

\newcommand{\phitriv}{\phi_{\triv}}

\DeclareMathOperator{\Spec}{Spec}

\DeclareMathOperator{\vol}{vol}

\DeclareMathOperator{\lsc}{lsc} 
\DeclareMathOperator{\ord}{ord}

\DeclareMathOperator{\PSH}{PSH}

\DeclareMathOperator{\redu}{red}

\DeclareMathOperator{\SNC}{SNC}
\DeclareMathOperator{\triv}{triv}
 
\DeclareMathOperator{\trdeg}{tr.deg} 
\DeclareMathOperator{\usc}{usc}

\newcommand{\retr}{p}

\newcommand{\D}{\Delta}

\newcommand{\simto}{\overset\sim\to}

\numberwithin{equation}{section}       

\newtheorem{prop} {Proposition} [section]
\newtheorem{thm}[prop] {Theorem} 
\newtheorem{defi}[prop] {Definition}
\newtheorem{lem}[prop] {Lemma}
\newtheorem{cor}[prop]{Corollary}
\newtheorem{prop-def}[prop]{Proposition-Definition}

\newtheorem*{thmA}{Theorem A} 
 
\newtheorem*{thmB}{Theorem B} 
\newtheorem*{thmC}{Theorem C} 
\newtheorem*{thmD}{Theorem D} 
 
\newtheorem*{thmE}{Theorem E} 
\newtheorem*{thmEp}{Theorem E'} 
\newtheorem*{thmF}{Theorem F}

\newtheorem{exam}[prop]{Example}
\newtheorem{rmk}[prop]{Remark}

\newtheorem{conj}[prop]{Conjecture}
\theoremstyle{remark}
\newtheorem*{ackn}{Acknowledgment}

\title{A non-Archimedean approach to K-stability}
\date{\today}

\author{S{\'e}bastien Boucksom \and Mattias Jonsson}

\address{CNRS--CMLS\\
  \'Ecole Polytechnique\\
  F-91128 Palaiseau Cedex\\
  France}
\email{boucksom@math.polytechnique.fr}

\address{Dept of Mathematics\\
  University of Michigan\\
  Ann Arbor, MI 48109-1043\\
  USA}
\email{mattiasj@umich.edu}

\begin{document}

\begin{abstract}
  We study K-stability properties of a smooth Fano variety $X$ 
  using non-Archi\-medean geometry, specifically the Berkovich analytification of
  $X$ with respect to the trivial absolute value on the ground field.
  More precisely, we view K-semistability and uniform K-stability as  
  conditions on the space of plurisubharmonic (psh) metrics on the anticanonical
  bundle of $X$. Using the non-Archimedean Calabi--Yau theorem and the 
  Legendre transform, this
  allows us to give a new proof that K-stability is equivalent to Ding stability. 
  By choosing suitable psh metrics, we also recover the valuative criterion of K-stability 
  by Fujita and Li. Finally, we study the asymptotic Fubini--Study operator, which associates
  a psh metric to any graded filtration (or norm) on the anticanonical ring. 
  Our results hold for arbitrary smooth polarized varieties, and 
  suitable adjoint/twisted notions of K-stability and Ding stability. They do
  not rely on the Minimal Model Program.
\end{abstract}

\maketitle

\setcounter{tocdepth}{1}
\tableofcontents
%
%
%
%
\section*{Introduction}\label{S101} 
%
%
%
%
Consider a polarized, smooth projective variety $(X,L)$ defined over a field $k$
of characteristic zero. For the purposes of this introduction, we identify $(X,L)$
with its analytification---in the sense of Berkovich---with respect to the \emph{trivial} 
absolute value on $k$. As a set, $X$ is thus the disjoint union of the spaces of real-valued valuations on the function field of each subvariety of $X$. Our goal in this paper is to use plurisubharmonic (psh) metrics on $L$
to study problems involving K-stability.

\smallskip
We start by recalling some notions and results from~\cite{trivval}. There is a canonical
map $L\to X$, with fibers being Berkovich affine lines. A metric on $L$ is a function on 
$L$ satisfying a certain homogeneity property on each fiber. 
Since $k$ is trivially valued, $L$ admits a canonical \emph{trivial metric}, which we
use to identify metrics on $L$ with \emph{functions} on $X$. 
For such metrics/functions, there are several natural positivity notions.

First we have the class $\cH(L)$ of \emph{positive metrics}~\cite{BHJ1}.
They are defined via ample \emph{test configurations}~\cite{Don02} for $(X,L)$.
In~\cite{trivval} they are called \emph{Fubini--Study metrics}, as 
each positive metric is built from a finite set of global sections of some multiple of $L$. 

By taking uniform limits of positive metrics we obtain the class of 
\emph{continuous psh metrics}, a notion going back to Zhang~\cite{Zha95}
and Gubler~\cite{GublerLocal}, and explored much more generally by 
Chambert-Loir and Ducros~\cite{CD12}, see also~\cite{GK15,GM16}.
By instead taking \emph{decreasing} 
limits of positive metrics, we obtain the larger class $\PSH(L)$ of psh metrics on $L$. 
This class has many good properties, as studied in~\cite{trivval}, see
also~\cite{siminag,trivval} for the discretely valued case.
As in the usual complex case, psh metrics are not necessarily continuous, and may even be \emph{singular}
in the sense of taking the value $-\infty$ at certain points in $X$.

The class $\PSH(L)$ contains several natural subclasses. First, we have the continuous psh
metrics. A larger class is formed by the \emph{bounded} psh metrics, \ie psh metrics
that are bounded as functions on $X$. This class is, in turn, contained in the class
$\cE^1(L)$ of psh metrics $\f$ of \emph{finite energy}, $E(\f)>-\infty$. Here $E$ is the 
\emph{Monge--Amp\`ere energy} functional. For positive metrics,
it can be defined in terms of intersection numbers, 
is monotonous, and hence has an extension to all of $\PSH(L)$, with 
values in $\R\cup\{-\infty\}$. 
Our thesis is that the class $\cE^1(L)$ is quite natural for
problems in K-stability, to be discussed in detail below.
The functional defined by $J(\f)=\sup_X\f-E(\f)$ acts
as an exhaustion function on $\cE^1(L)/\R$. In particular, 
$J(\f)\ge0$, with equality iff $\f$ is constant.

The \emph{Monge--Amp\`ere operator} assigns a Radon probability 
measure $\MA(\f)$ on $X$ to any metric $\f\in\cE^1(L)$. It is continuous under monotone
limits and extends the operator introduced by Chambert-Loir~\cite{CL06} and
Gubler~\cite{GublerTrop}. The \emph{Calabi--Yau} theorem~\cite{nama,trivval} 
asserts that we have a bijection $\MA\colon\cE^1(L)/\R\to\cM^1(X)$, where $\cM^1(X)$
is the space of Radon probability measures $\mu$ of \emph{finite energy}, \ie such that
$$
E^*(\mu):=\sup_{\f\in\cE^1(L)}\left(E(\f)-\int\f\,d\mu\right)<\infty.
$$
%
%
%
%
\subsection*{K-stability and Ding stability}
The notion of K-stability was introduced by Yau, Tian and Donaldson,
as a conjectural criterion for the existence of special metrics in K\"ahler geometry;
this conjecture was proved in the Fano case, see~\cite{CDS15,Tian15}
and also~\cite{DaSz16,BBJ15,CSW15}.
K-stability involves studying the sign of the \emph{Donaldson--Futaki functional}
on the space of ample test configurations for $(X,L)$, 
\ie on the space $\cH(L)$. 
In~\cite{BHJ1} we showed that K-stability can
be expressed in terms of a modified functional, the (non-Archimedean) \emph{Mabuchi functional}, that has better properties with respect to base change, and which 
naturally extends to a functional $M:\cE^1(L)\to\R\cup\{+\infty\}$. 

In the Fano case, when $L=-K_X$, the Mabuchi functional factors through the
Monge--Amp\`ere operator as 
$$
M(\f)=\Ent\left(\MA(\f)\right)-E^*\left(\MA(\f)\right),
$$
where $E^*(\mu)$ is the energy of $\mu\in\cM^1(X)$, whereas
$\Ent(\mu):=\int_XA\,d\mu$ is the \emph{entropy} of $\mu$, defined as the integral
of the \emph{log discrepancy} function $A\colon X\to\R_+\cup\{+\infty\}$.

Another functional, the (non-Archimedean) \emph{Ding functional} is also useful. 
First defined in~\cite{Berm16}, it can be written as 
$D=L-E$, where $E$ is the Monge--Amp\`ere energy, and $L$ is the Legendre 
transform of entropy: $L(\f)=\inf_X(\f+A)$, with $A$ the log discrepancy.

In this language, a Fano manifold $X$ is K-semistable (resp.\ uniformly K-stable) if
$M\ge 0$ (resp.\ there exists $\e>0$ such that $M\ge\e J$) on $\cH(L)$.
Similarly, $X$ is Ding-semistable (resp.\ uniformly Ding-stable) if the corresponding
inequalities hold with the Mabuchi functional $M$ replaced by the Ding functional $D$.
It was proved in~\cite{BBJ15} (see also~\cite{FujitaValcrit}), 
using techniques from the Minimal Model Program along the same lines as \cite{LX14}, 
that Ding semistability is equivalent to K-semistability, and similarly for the uniform
versions. 

Here we give a new proof of (a version of) 
the equivalence of Ding-stability and K-stability. 
Let $L$ be any ample line bundle on $X$. The definitions of the functionals $M$ and $D$
still make sense; we call them the \emph{adjoint Mabuchi functional} 
and \emph{adjoint Ding functional}, respectively.
We say that $(X,L)$ is \emph{K-semistable in the adjoint sense} if the adjoint 
Mabuchi functional is nonnegative on $\cE^1(L)$.
Similarly, we define adjoint versions of Ding semistability, and 
uniform Ding and K-stability. 

Outside the Fano case, the adjoint Mabuchi functional
differs from the Mabuchi functional discussed in the beginning of this section. Up to a conjectural approximation result (Conjecture~\ref{Conj302}), adjoint K-stability is equivalent to Dervan's notion of \emph{twisted K-stability} \cite{Der16} in the `twisted Fano case', \ie when the (not necessarily semipositive) twisting class $T$ is defined so that $L=-(K_X+T)$, see \S\ref{sec:twisted}. 
\begin{thmA}
  For any ample line bundle $L$ on a smooth projective variety $X$, we have:
  \begin{itemize}
  \item[(i)]
    $L$ is K-semistable in the adjoint sense iff it is Ding-semistable in the adjoint sense;
  \item[(ii)]
    $L$ is uniformly K-stable in the adjoint sense iff it is uniformly Ding-stable 
    in the adjoint sense.
  \end{itemize}
\end{thmA}

The proof is an adaptation to the non-Archimedean case of ideas of
Berman~\cite{BermThermo}, who was inspired by thermodynamics. It is based 
on the non-Archimedean Calabi--Yau theorem together with the fact that 
the functionals $L$ and $E$ on $\cE^1(L)$ are the Legendre transforms 
of the functionals $\Ent$ and $E^*$, respectively, on $\cM^1(X)$.

The adjoint stability notions above were defined in terms of metrics of finite
energy. This is because the Calabi--Yau theorem is an assertion about 
metrics and measures of finite energy. It is natural to ask whether the adjoint stability
notions remain unchanged if we replace $\cE^1(L)$ by the space $\cH(L)$
(\ie ample test configurations). The answer is `yes' for adjoint Ding stability, since the Ding
functional is continuous under decreasing limits. In the Fano case, the answer
is also `yes' by~\cite{BBJ15,FujitaValcrit}. We expect the answer to be `yes' in general,
even though the Mabuchi functional is not continuous under decreasing limits.
%
%
%
%
\subsection*{Filtrations, norms, and metrics}
Donaldson, Sz\'ekelyhidi and others have suggested to
strengthen the notion of K-stability by allowing more general objects than
test configurations. One generalization is the class $\cE^1(L)$ above.
Another one, explored by Sz\'ekelyhidi~\cite{Sze15}, is given by
(graded) \emph{filtrations} of the section ring $R=R(X,L)$. 
Such filtrations are (see~\S\ref{s:filtnorm}) 
in 1--1 correspondence with \emph{graded norms} $\|\cdot\|_\bullet$ on $R$,
\ie the data of a non-Archimedean $k$-vector space norm $\|\cdot\|_m$ on $R_m:=H^0(X,mL)$ 
for each $m\ge1$ such that $\|s\otimes s'\|_{m+m'}\le\|s\|_m\|s'\|_{m'}$ 
for $s\in R_m$, $s'\in R_{m'}$. 
A graded norm is \emph{bounded} if there exists $C\ge1$ such that
$C^{-m}\le\|\cdot\|_m\le C^m$ on $R_m\setminus\{0\}$ for all $m$.
To such a graded norm we associate a bounded psh metric on $L$,
the \emph{asymptotic Fubini--Study metric} $\FS(\|\cdot\|_\bullet)$. 

Our next result characterizes the range of the asymptotic Fubini--Study operator.
A bounded metric $\f\in\PSH(L)$ is \emph{regularizable from below}
if it is the limit of an increasing net of positive metrics. For example,
any continuous psh metric is regularizable from below.
The analogous notion on domains in $\C^n$ was studied by Bedford~\cite{Bed80}.
\begin{thmB}
  A psh metric on $L$ lies in the image of the asymptotic Fubini--Study operator iff
  it is regularizable from below.
\end{thmB}
It follows from the construction of the asymptotic Fubini--Study operator that any metric in 
its image is regularizable from below.
To prove Theorem~B we construct a one-sided inverse,
the \emph{supremum graded norm} $\|\cdot\|_{\f,\bullet}$ of any bounded $\f\in\PSH(L)$. 
The formula $\f=\FS(\|\cdot\|_{\f,\bullet})$ can be viewed as a non-Archimedean 
$L^\infty$-version of Bergman kernel asymptotics.

\smallskip
The asymptotic Fubini--Study operator is not injective in general, and we study
the lack of injectivity.
To each bounded graded norm is associated a \emph{limit measure}~\cite{BC11},
a probability measure on $\R$ that describes the asymptotic distribution of 
norms of vectors in $R_m$ as $m\to\infty$.
More generally, Chen and Maclean~\cite{Chen-Maclean} showed that 
any two bounded graded norms on $R$ 
induce a \emph{relative limit measure}. The second moment of 
this measure (\ie of the 
corresponding random variable) defines a semidistance $d_2$ 
on the space of bounded graded norms. More generally, there is a 
semidistance $d_p$ for any $p\in[1,\infty)$. One can show that 
two bounded graded norms are at $d_p$-distance zero for some $p$
iff they are so for all $p$; in this case the graded norms are called
\emph{equivalent}. For $p=2$, the semidistance $d_2$ on bounded graded norms coincides with the limit pseudo-metric introduced in \cite{Cod18}. 
\begin{thmC}
  Two bounded graded norms induce the same associated Fubini--Study metric
  iff they are equivalent.
\end{thmC}
This theorem is proved by exhibiting a formula for the $d_1$-semidistance:
\begin{equation}\label{e433}
  d_1(\|\cdot\|_\bullet,\|\cdot\|'_\bullet)
  =E(\f,Q(\f\wedge\f'))+E(\f',Q(\f\wedge\f')),
\end{equation}
where $\f:=\FS(\|\cdot\|_\bullet)$, $\f':=\FS(\|\cdot\|'_\bullet)$, and
$Q(\f\wedge\f')$ is the largest psh metric that is regularizable from below
and dominated by $\f\wedge\f':=\min\{\f,\f'\}$.
In fact, the right-hand side of~\eqref{e433} defines a distance on
the space $\PSH^\uparrow(L)$ of psh metrics regularizable from below.
This is analogous to the \emph{Darvas distance} in the 
complex analytic case~\cite{Dar15}. The asymptotic Fubini--Study operator
now becomes an isometric bijection between the space of equivalence 
classes of bounded graded norms and the space $\PSH^\uparrow(L)$.

To prove~\eqref{e433}, the main step is to establish the equality
\begin{equation}\label{e434}
  E(\f,\f')
  =\vol(\|\cdot\|_\bullet,\|\cdot\|'_\bullet),
\end{equation}
where the \emph{relative volume} $\vol(\|\cdot\|_\bullet,\|\cdot\|'_\bullet)$
is the barycenter of the relative limit measure.
By taking the two graded norms as supremum norms of continuous metrics on $L$,
we recover a version of the main results of~\cite{BE18,BGJKM16} in the 
trivially valued case.

\smallskip
The $L^2$-norm of a (graded) filtration introduced by Sz\'ekelyhidi
is equal to the variance of the relative limit measure with respect to the trivial
graded filtration.
It follows that \emph{a filtration has $L^2$-norm zero
iff its associated Fubini--Study metric is constant}.
Sz\'ekelyhidi also defined a notion of (Donaldson--)Futaki invariant of a
(graded) filtration $\cF$. We show that if $X$ is uniformly K-stable,
then the Donaldson--Futaki invariant is strictly positive for every filtration of 
positive norm.

%
%
%
%
\subsection*{A valuative criterion for adjoint K-stability}
Next we study the valuative criterion for K-stability of Fujita~\cite{FujitaValcrit} 
and Li~\cite{LiEquivariant} using psh metrics.
There is a subset 
$\Xval\subset X$ consisting of \emph{valuations} of the function field of $X$ (the latter being in fact the disjoint union of $\Yval$, as $Y$ ranges over irreducible 
subvarieties of $X$). To each point $x\in\Xval$, we can associate several invariants.
First we have the log discrepancy $A(x)\in[0,+\infty]$. 
Second, given an ample line bundle $L$ on $X$, the valuation
$x$ defines a graded norm on the section ring $R=R(X,L)$, see~\cite{BKMS}. 
When bounded, this norm induces a limit measure on $\R$, whose 
barycenter is denoted by $S(x)\in[0,+\infty)$, and 
can be viewed as the \emph{expected vanishing order} of elements of $R$ along $x$. 
When the filtration is unbounded, we set $S(x)=+\infty$.
\begin{thmD}
  For any point $x\in\Xval$, we have $S(x)=E^*(\d_x)$, 
  the energy of the Dirac mass at $x$. We have $S(x)=\infty$ iff
  the point $x$ is pluripolar, \ie there exists $\f\in\PSH(L)$ with
  $\f(x)=-\infty$. If $S(x)<\infty$,
  then the unique solution to the Monge--Amp\`ere equation 
  \begin{equation*}
    \MA(\f_x)=\d_x
  \end{equation*}
  normalized by $\f_x(x)=0$ is continuous.
\end{thmD}

\smallskip
The adjoint Mabuchi functional can be written as 
$\Mad(\f)=\Ent(\MA(\f))-E^*(\MA(\f))$.
By the Calabi--Yau theorem, $L$ is therefore K-semistable in the adjoint sense
iff $\Ent(\mu)\ge E^*(\mu)$ for all $\mu\in\cM^1(X)$. 
Note also that $\Mad(\f_x)=A(x)-S(x)$. 
By studying the entropy and energy functionals, we prove
\begin{thmE}
  The following conditions are equivalent:
  \begin{itemize}
  \item[(i)]
    $L$ is K-semistable in the adjoint sense;
  \item[(ii)]
    $\Mad(\f_x)\ge 0$ for all nonpluripolar $x\in\Xval$;
  \item[(iii)]
    $A(x)\ge S(x)$ for all nonpluripolar $x\in\Xval$;
  \item[(iv)]
    $A(x)\ge S(x)$ for all divisorial valuations $x\in\Xval$.
  \end{itemize}
\end{thmE}
\begin{thmEp}
  The following conditions are equivalent:
  \begin{itemize}
  \item[(i)]
    $L$ is uniformly K-stable in the adjoint sense;
  \item[(ii)]
    there exists $\e>0$ such that 
    $\Mad(\f_x)\ge\e J(\f_x)$ for all nonpluripolar $x\in\Xval$;
  \item[(iii)] 
    there exists $\e>0$ such that 
   $A(x)\ge(1+\e)S(x)$ for all for all nonpluripolar $x\in\Xval$;
  \item[(iv)]
    there exists $\e>0$ such that 
    $A(x)\ge (1+\e)S(x)$ for all divisorial valuations $x\in\Xval$.
  \end{itemize}
\end{thmEp}
Theorem~E and Theorem~E' generalize Li and Fujita's valuative criterion~\cite{LiEquivariant,FujitaValcrit}
for K-stability to the adjoint setting. Our proof is completely different from theirs.
However, we should mention that Fujita also proves that in~(iv), it suffices
to consider ``dreamy'' valuations $x$. These are points $x\in\Xdiv$ such that 
$\f_x\in\cH(L)$.
%
%
%
%
\subsection*{Adjoint K-stability and uniform K-stability}
Above we have used adjoint notions of K-semistability and uniform K-stability.
As an intermediate notion, we say that $L$ is \emph{K-stable} in the adjoint sense
if $\Mad(\f)\ge0$ for all $\f\in\cE^1(L)$, with equality iff $\f$ is constant.
\begin{thmF}
  If $k=\C$, then $L$ is K-stable in the adjoint sense iff $L$ is uniformly K-stable
  in the adjoint sense.
\end{thmF}
This relies on~\cite[Theorem~E]{BlumJonsson}, which guarantees
that the ratio $A(x)/S(x)$ attains its minimum at some nonpluripolar 
point $x\in \Xval$.

In the Fano case, it follows from~\cite{CDS15} and~\cite{BHJ2} that K-stability
is equivalent to uniform K-stability. However, there seems to be no 
algebraic proof of this fact. We would get such a proof if we knew that the point 
$x$ above was a  ``dreamy'' divisorial valuation in the sense of Fujita, \ie 
$\f_x\in\cH(L)$.
%
%
%
%
\subsection*{Approaches to K-stability}
We have tried to demonstrate that psh metrics form a useful tool for 
studying K-stability. They encompass not only test configurations, but also
graded filtrations/norms on the section ring, and nonpluripolar valuations.
Another way of studying K-stability of Fano varieties 
was introduced by Chi Li~\cite{Li15,LiEquivariant} and involves considering
the cone over the variety. The cone point is a klt singularity, and 
there has been much recent activity on the study of general klt 
singularities~\cite{Blu16,LL16,LX17,BlumLiu}, but we shall not discuss this further here.
%
%
%
%
\subsection*{Organization of paper}
After reviewing some background material from~\cite{trivval} in~\S\ref{S102} we
start in~\S\ref{sec:adjoint} by proving Theorem~A on the equivalence of adjoint K-stability
and Ding-stability. Then, in~\S\S\ref{s:filtnorm}--\ref{sec:asFS}, we study the 
relationship between graded norms/filtrations and psh metrics, proving 
Theorems~B and~C. Finally, the valuative criterion for K-stability
(Theorems~D,~E and~E') is proved in~\S\ref{S405}, as is Theorem~F.
%
%
%
%
\begin{ackn}
  We thank E.~Bedford, R.~Berman, H.~Blum, G.~Codogni, R.~Dervan, A.~Ducros, C.~Favre, T.~Hisamoto,  C.~Li, J.~Poineau 
  and M.~Stevenson for fruitful discussions and comments.
  The first author was partially supported by the ANR grant GRACK.\@
  The second author was partially supported by NSF grant
  DMS-1600011 and the United States---Israel Binational Science Foundation.
\end{ackn}
%
%
%
%
%
%
\section{Background}\label{S102}
For details on the material in this section, see~\cite{trivval}.
%
%
%
%
\subsection{Setup}
Throughout the paper, $k$ is field of characteristic zero, equipped with the 
\emph{trivial} absolute value.
By a $k$-variety we mean an integral separated scheme of finite type over $k$. 
We fix a logarithm $\log\colon\R_+^\times\to\R$,  with inverse $\exp\colon\R\to\R_+^\times$.
%
%
%
%
\subsection{Analytification}
The analytification functor in~\cite[\S 3.5]{BerkBook}
associates to any $k$-variety a $k$-analytic space in the sense of Berkovich. 
Typically we write
$X$ for the analytification and $\Xsch$ for the underlying variety, viewed as a scheme.
Let $\A^n=\A^n_k$, $\P^n=\P^n_k$ and $\G_m=\G_{m,k}$ be the
analytifications of $\bbA^n=\bbA^n_k$, $\bbP^n=\bbP^n_k$, $\bbG_m=\bbG_{m,k}$,
respectively.

The analytification $X$ of $\Xsch$ consists of all 
pairs $x=(\xi,|\cdot|)$, where $\xi\in\Xsch$ 
is a point and $|\cdot|=|\cdot|_x$ is a multiplicative norm on the residue field 
$\kappa(\xi)$ extending the trivial norm on $k$.
We denote by $\cH(x)$ the completion of $\kappa(\xi)$ with respect to this norm. 
The surjective map $\ker\colon X\to\Xsch$ sending $(\xi,|\cdot|)$ to $\xi$ is 
called the \emph{kernel} map. The points in $X$ whose kernel is the generic point of 
$\Xsch$ are the valuations
of the function field of $\Xsch$ that are trivial on $k$. 
They form a subset $\Xval\subset X$. 

There is a section $\Xsch\hookrightarrow X$ of the kernel map,
defined by associating to $\xi\in\Xsch$ the point in $X$
defined by the trivial norm on $\kappa(\xi)$.
The image of the generic point of $\Xsch$ is called the generic point of $X$:
it corresponds to the trivial valuation on $k(X)$.

If $\Xsch=\Spec A$ is affine, with $A$ a finitely generated $k$-algebra, 
$X$ consists of all multiplicative seminorms on $A$ extending the trivial norm on $k$.

The \emph{Zariski topology} on $X$ is the weakest topology in which $\ker\colon X\to\Xsch$ is
continuous. We shall work in \emph{Berkovich} topology, 
the coarsest refinement of the Zariski topology for which 
the following holds: for any open affine $\cU=\Spec A\subset\Xsch$ and any 
$f\in A$, the function $\ker^{-1}(\cU)\ni x\to|f(x)|\in\R_+$ is continuous, where
$f(x)$ denotes the image of $f$ in $k(\xi)\subset\cH(x)$, so that 
$|f(x)|=|f|_x$. 
The subset $\Xval\subset X$ is dense.
In general, $X$ is Hausdorff, locally compact,  and locally path connected;
it is compact iff $\Xsch$ is proper.
We say $X$ is projective (resp.\ smooth) if $\Xsch$ has the 
corresponding properties.

When $X$ is compact, there is a \emph{reduction map} $\redu\colon X\to\Xsch$,
defined as follows.
Let $x\in X$ and set $\xi:=\ker x\in\Xsch$, so that $x$ defines valuation on $\kappa(\xi)$
that is trivial on $k$. If $\cY$ is the closure of $\xi$ in $\Xsch$, then 
$\eta=\redu(x)\in\cY\subset\Xsch$ is the unique point such that 
$|f(x)|\le 1$ for $f\in\cO_{\cY,\eta}$ and $|f(x)|<1$
when further $f(\eta)=0$.
%
%
%
%
\subsection{Divisorial and quasimonomial points}
If $x\in\Xval$, let $s(x):=\trdeg(\widetilde{\cH(x)}/k)$ be the transcendence degree
of $x$ and $t(x):=\dim_\Q\sqrt{|\cH(x)^\times|}$ the rational rank.
The \emph{Abhyankar inequality} says that $s(x)+t(x)\le\dim X$.
A point $x$ is \emph{quasimonomial} if equality holds.
It is \emph{divisorial} if $s(x)=\dim X-1$ and $t(x)=1$. 
We write $\Xqm$ and $\Xdiv$ for the set of quasimonomial and divisorial points,
respectively. Thus $\Xdiv\subset\Xqm\subset\Xval$.

\emph{From now on, assume $X$ is smooth and projective}.
Then any divisorial point is of the form 
$\exp(-c\ord_E)$, where $c>0$ and $E$ is a prime divisor on a smooth 
projective variety $\cY$ that admits a proper birational map onto $\Xsch$.

Similarly, quasimonomial points can be geometrically described
as follows, see~\cite{jonmus}.
A \emph{log smooth pair} $(\cY,\cD)$ over $\Xsch$ is the data of a smooth $k$-variety $\cY$
together with a proper birational morphism $\pi\colon\cY\to\Xsch$,
and $\cD$ a reduced simple normal crossings divisor on $\cY$
such that $\pi$ is an isomorphism outside the support of $\cD$. 
To such a pair we associate a \emph{dual cone complex} $\D(\cY,\cD)$, the cones 
of which are in bijection with strata 
(\ie connected components of intersections of irreducible components) of $\cD$.
We embed $\D(\cY,\cD)$ into $\Xval$ as the set of monomial points with respect to 
local equations of the irreducible components of $\cD$ 
at the generic point of the given stratum. 
The apex of $\D(\cY,\cD)$ is the generic point of $X$.
We have $x\in\Xqm$ iff $x\in\D(\cY,\cD)$
for some log smooth pair $(\cY,\cD)$ over $\Xsch$.

For many purposes, it is better to describe divisorial and quasimonomial points 
using snc test configurations, see below.
%
%
%
%
\subsection{Test configurations}
A \emph{test configuration} $\cX$ for $X$ consists of the following data: 
\begin{itemize}
\item[(i)] a flat and projective morphism of $k$-schemes $\pi:\cX\to\bbA^1_k$;
\item[(ii)] a $\bbG_m$-action on $\cX$ lifting the canonical action on $\bbA^1_k$;  
\item[(iii)] an isomorphism $\cX_1\simto X$. 
\end{itemize}
The \emph{trivial} test configuration for $X$ is given by the product 
$\Xsch\times\bbA^1_k$, with the trivial $\bbG_m$-action on $\Xsch$.
Given test configurations $\cX$, $\cX'$ for $X$ there exists a unique 
$\bbG_m$-equivariant birational map $\cX'\dashrightarrow\cX$ extending the 
isomorphism $\cX'_1\simeq X\simeq\cX_1$. We say that $\cX'$ 
\emph{dominates} $\cX$ if this map is a morphism.
Any two test configurations can be dominated by a third. 
Two test configurations that dominate each other will be identified.

For any test configuration $\cX$ for $X$, we have a \emph{Gauss embedding}
$\sigma_\cX\colon X\to\cXan$, whose image consists of all $k^\times$-invariant 
points $y\in\cXan$ satisfying $|\unipar(y)|=\exp(-1)$, where $\unipar$ is
the coordinate on $\bbA^1$. Each such point has a well-defined reduction
$\redu_\cX(y)\in\cX_0$. Somewhat abusively, we denote the composition
$\redu_\cX\circ\sigma_\cX\colon X\to\cX_0$ by $\redu_\cX$.
When $\cX$ is the trivial test configuration, $\cX_0\simeq X$, and $\redu_\cX$
coincides with the reduction map considered earlier.
%
%
%
%
\subsection{Snc test configurations}
An \emph{snc test configuration} is a test configuration $\cX$
that dominates the trivial test configuration and whose 
central fiber $\cX_0$ has strict normal crossing support. 
By Hironaka's theorem, the set $\SNC(X)$ of snc test configurations
is directed and cofinal in the set of all test configurations.

To any snc test configuration $\cX$ of $X$ is associated a \emph{dual complex} 
$\D_\cX$, a simplicial complex whose simplices are in 1-1 correspondence with 
strata of $\cX_0$, \ie connected components of nonempty intersections of 
irreducible components of $\cX_0$. 
For example, the dual complex of the trivial test configuration has a single vertex, 
corresponding to single stratum $\Xsch\times\{0\}$.

We can view $(\cX,\cX_0)$ as a log smooth pair over $\cX$, whose dual cone complex
$\D(\cX,\cX_0)$ is the cone over the dual complex $\D_\cX$. The $\bbG_m$-equivariance
of $\cX$ implies that the image of $\D(\cX,\cX_0)$ in $\cXan$ consists of 
$k^\times$-invariant points. We identify $\D_\cX$ with the subset of 
$\D(\cX,\cX_0)\subset\cXan$ cut out by the equation $|\unipar|=\exp(-1)$. 
Via the Gauss embedding 
$\sigma_\cX\colon X\to\Xan$, this allows us to view the dual complex
$\D_\cX$ as a subset of $X$. The points in $\D_\cX$ are all quasimonomial,
and every quasimonomial point belongs to some dual complex.
Similarly we define a \emph{retraction} $\retr_\cX\colon X\to\D_\cX$. 
The directed system $(\retr_\cX)_{\cX\in\SNC(X)}$ induces a homeomorphism
$X\simto\varprojlim\D_\cX$.
%
%
%
%
\subsection{Metrics on line bundles}
We view a line bundle $L$ on $X$ as a $k$-analytic
space with a projection $p\colon L\to X$. Its fiber over $x\in X$ is isomorphic
to the Berkovich affine line over the complete residue field $\cH(x)$.
Write $L^\times$ for $L$ with the zero section removed.
We typically work additively, so that a metric on $L$ is function $\phi\colon L^\times\to\R$
such that $|\cdot|_\phi:=\exp(-\phi)\colon L^\times\to\R_+^\times$ 
behaves like a norm on each fiber $p^{-1}(x)$.

If $\phi$ is a metric on $L$, any other metric is of the form $\phi+\f$, 
where $\f$ is a function on $X$, that is, a metric on $\cO_X$.
If $\phi_i$ is a metric on $L_i$, $i=1,2$, then $\phi_1+\phi_2$
is a metric on $L_1+L_2$.
If $\phi$ is a metric on $L$, then $-\phi$ is a metric on $-L$. 
If $\phi$ is a metric on $mL$, where $m\ge 1$, then $m^{-1}\phi$ is a metric on $L$.
If $\phi_1,\phi_2$ are metrics on $L$, so is $\max\{\phi_1,\phi_2\}$.

If $\phi_1$ and $\phi_2$ are metrics on $L$ inducing the same metric $m\phi_1=m\phi_2$
on $mL$ for some $m\ge 1$, then $\phi_1=\phi_2$. We can therefore define a metric on a $\Q$-line
bundle $L$ as the choice of a metric $\phi_m$ on $mL$ for $m$ sufficiently divisible, 
with the compatibility condition $l\phi_m=\phi_{ml}$.
%
%
%
%
\subsection{Metrics as functions}
Any line bundle $L$ on $X$
admits a \emph{trivial metric} $\phitriv$ defined as follows. Given a point $x\in X$,
set $\xi:=\redu(x)\in\Xsch$, and let $t$ be a nonvanishing section of $L$ on an open
neighborhood $\cU\subset\Xsch$ of $\xi$.
Then $t$ defines a nonvanishing analytic section of $L$ on the Zariski open neighborhood 
$\cUan$ of $x$ in $X$, and $\phitriv(t(x))=0$.

As multiplicative notation for the trivial metric, we use $|\cdot|_0=\exp(-\phitriv)$. 
Given a global section $s$ of $L$, the function $|s|_0$ on $X$ has the following
properties. Given $x\in X$, pick $\cU$ and $t$ as above, and write $s=ft$, where 
$f\in\Gamma(\cU,\cO_{\Xsch})$. Then $|s|_0(x)=|f(x)|$.

The trivial metric allows us to think of metrics on $L$ as \emph{functions} on $X$,
and we shall frequently do so in what follows. 
Indeed, if $\phi$ is a metric on $L$, then $\f:=\phi-\phitriv$ is a function on $X$.
In this way, the trivial metric becomes the zero function.
%
%
%
%
\subsection{Metrics from test configurations}
A \emph{test configuration} for a $\Q$-line bundle $L$ on $X$ consists of a 
test configuration $\cX$ for $X$ together with the following data:
\begin{itemize}
\item[(iv)] a $\bbG_m$-linearized $\Q$-line bundle $\cL$ on $\cX$;
\item[(v)] an isomorphism $\cL|_{\cX_1}\simeq L$.
\end{itemize}
If more precision is needed, we say that $(\cX,\cL)$ is a test configuration for $(X,L)$.

Given a $\bbG_m$-action on $(\Xsch,\Lsch)$ we have a \emph{product test configuration}
$(\Xsch\times\bbA^1,\Lsch\times\bbA^1)$, with diagonal action.
If the action on $(\Xsch,\Lsch)$ is trivial, we obtain the 
\emph{trivial} test configuration for $(X,L)$.
A test configuration $(\tcX,\tcL)$ \emph{dominates} another test configuration $(\cX,\cL)$ 
for $(X,L)$, if $\tcX$ dominates $\cX$ and $\tcL$ is the pullback of $\cL$ under $\tcX\to\cX$.

Any test configuration $\cL$ for $L$ induces a metric $\f_\cL$ 
(viewed as function on $L^\times$) on $L$ as follows.
First assume $\cL$ is a line bundle. Using the Gauss embedding 
$\sigma_\cL\colon L\to\cLan$ it suffices to define a metric $\f_\cL$ on $\cLan$.
Pick any point $y\in\cXan$ with $|\unipar(y)|=\exp(-1)$, write
$\xi:=\redu_\cX(y)\in\cX_0$, and let $s$ be a section of $\cL$ 
on a Zariski open neighborhood of $\xi$ in $\cX$ such that $s(\xi)\ne0$. 
Then $s$ defines a section of $\cLan$ on a Zariski open neighborhood of $y$,
and we declare $\f_\cL(s(y))=0$.
One checks that this definition does not depend on the choice of $s$. 
Further, $\f_{m\cL}=m\f_\cL$, which allows us to define $\f_\cL$ when $\cL$
is a $\Q$-line bundle.
The metric $\f_L$ does not change when replacing $\cL$ by a pullback. 
When $\cL$ is the trivial test configuration for $L$, $\phi_\cL$ is 
the trivial metric on $L$.

As in~\cite{BHJ1}, the restriction of the function $\f=\f_\cL$ to $\Xdiv$
can be described as follows.
We may assume $\cL$ dominates the trivial test configuration 
via $\rho\colon\cX\to\Xsch\times\bbA^1$. Then 
$\cL-\rho^*(\Lsch\times\bbA^1)=\cO_{\cX}(D)$ for a 
divisor $D$ on $\cX$ cosupported on $\cX_0$, and
\begin{equation*}
  \f_\cL(v_E)=\frac{\ord_E(D)}{\ord_E(\cX_0)}
\end{equation*}
for every irreducible component
$E$ of $\cX_0$, where $v_E\in\Xdiv$ is the restriction of $\ord_E/\ord_E(\cX_0)$ 
from $k(\cX)$ to $k(X)$.
Varying $\cX$ determines $\f|_{\Xdiv}$ completely.
%
%
%
%
\subsection{Positive metrics}\label{S401}
Suppose $L$ is ample.
A test configuration $\cL$ for $L$ is \emph{ample} (resp.\ semiample)
if $\cL$ is relatively ample (resp.\ semiample) for the canonical morphism $\cL\to\bbA^1$.
A metric on $L$ is \emph{positive} if it is defined by a semiample test configuration for $L$. 
Every positive metric is in fact associated to a unique normal ample
test configuration for $L$ (which may not dominate the trivial test configuration).
A positive metric is the same thing
as a \emph{Fubini--Study metric} (or FS metric), \ie
a metric $\f$ (viewed as function on $X$) of the form
\begin{equation}\label{e201}
  \f:=m^{-1}\max_\a(\log|s_\a|_0+\la_\a),
\end{equation}
where $m\ge1$,  $(s_\a)$ is a finite set of global sections of $mL$
without common zero, and $\la_\a\in\Z$. 

Denote by $\cH(L)_\R$ the set of metrics of the form~\eqref{e201}
with $\la_\a\in\R$ for each $\a$. Such metrics are continuous and can be uniformly
approximated by positive metrics.
%
%
%
%
\subsection{DFS metrics}
A \emph{DFS metric} on a $\Q$-line bundle $L$ on $X$ is a metric of the form 
$\phi_1-\phi_2$, with $\phi_i$ an FS metric on $L_i$, $i=1,2$, where 
$L=L_1-L_2$. Equivalently, a DFS metric is a metric defined by a test configuration for $L$.
The set of DFS metrics on $L$ is dense in the space of continuous metrics in 
the topology of uniform convergence. It plays the role of smooth metrics in complex 
geometry.
%
%
%
%
\subsection{Monge--Amp\`ere operator and energy functionals}
As a special case of the theory of Chambert-Loir and 
Ducros~\cite{CD12}, there is a \emph{mixed Monge--Amp\`ere operator}
that associates to any DFS metrics $\phi_1,\dots,\phi_n$ on line bundles
$L_1,\dots,L_n$ on $X$, a signed finite atomic measure
$dd^c\phi_1\wedge\dots\wedge dd^c\phi_n$ on $X$, supported on 
divisorial points, and of mass $(L_1\cdot\ldots\cdot L_n)$.
This measure is positive if the $L_i$ are positive metrics. 

As above, we think of the metrics as functions 
$\f_i=\phi_i-\phitriv$ on $X$, and write
\begin{equation*}
  dd^c\phi_i=\omega_{\f_i}=\omega+dd^c\f_i,
\end{equation*}
so that the mixed Monge--Amp\`ere measure becomes
$\omega_{\f_1}\wedge\dots\wedge\omega_{\f_n}$.
Here $\omega$ and $\omega_{\f_i}$ can be viewed as currents~\cite{CD12}
or $\delta$-forms~\cite{GK17} on $X$, but we do not need this terminology.

Now suppose $L_1=\dots=L_n=:L$ is an ample $\Q$-line bundle.
We then write 
\begin{equation*}
  \MA(\f):=V^{-1}\omega_\f^n,
\end{equation*}
where $V=(L^n)$. If $\f\in\cH(L)$, this is a probability measure on $X$.
Since $\f=0$ corresponds to the trivial metric on $L$, 
$\MA(0)=\omega^n$ is a Dirac mass at the generic point of $X$.

If $\f$ and $\p$ are DFS metrics on $L$, the 
\emph{Monge--Amp\`ere energy} of $\f$ with respect to $\p$ is defined by 
\begin{equation*}
  E(\f,\p)=\frac{1}{(n+1)V}\sum_{j=0}^n\int(\f-\p)\omega_\f^j\wedge\omega_\p^{n-j}.
\end{equation*}
In addition to the functional $E$, we set
\begin{equation*}
  I(\f,\p)=\int(\f-\p)(\MA(\p)-\MA(\f))
  \quad\text{and}\quad
  J_\p(\f)=\int(\f-\p)\MA(\p)-E(\f,\p).
\end{equation*}
On positive metrics, the functionals $I$, $J$, and $I-J$ are nonnegative and comparable:
\begin{equation}\label{e410}
  n^{-1}J\le I-J\le nJ.
\end{equation}
The functionals $E$, $I$ and $J$ naturally have two arguments. Often it is convenient to 
fix the second argument $\p$ as the trivial metric, and write
$E(\f)$, $I(\f)$ and $J(\f)$.
%
%
%
%
\subsection{Psh metrics}
For the rest of this section, $L$ is an ample $\Q$-line bundle.
A \emph{psh} (plurisubharmonic) metric on $L$ is the pointwise limit of any
decreasing net of positive metrics, provided the limit is $\not\equiv-\infty$. We denote
by $\PSH(L)$ the set of all psh metrics. 

By subtracting the trivial metric, we can view the elements of $\PSH(L)$
as functions on $X$ with values in $\R\cup\{-\infty\}$. 
This is analogous to the notion of $\omega$-psh functions in complex geometry;
we can think of $\omega$ as the curvature form of the trivial metric,

We equip $\PSH(L)$ with the topology of pointwise convergence on $\Xqm$. 
If $\f\in\PSH(L)$ and $x\in X$, then the net $(\f\circ\retr_\cX(x))_{X\in\SNC(X)}$ 
is decreasing, with limit $\f(x)$.
Any $\f\in\PSH(L)$ takes its maximum value at the generic point of $X$.
If $\f\in\PSH(L)$, then $\f+c\in\PSH(L)$ for all $c\in\R$. 
If $\f,\p\in\PSH(L)$, then $\f\vee\p:=\max\{\f,\p\}\in\PSH(L)$.

If $(\f_j)_j$ is a decreasing net in $\PSH(L)$, and $\f$ is the pointwise limit of $(\f_j)$,
then $\f\in\PSH(L)$, or $\f_j\equiv-\infty$. If instead $(\f_j)_j$ is an
increasing net that is bounded from above, then $\f_j$ converges in $\PSH(L)$
to some $\f\in\PSH(L)$. Thus $\f=\lim_j\f_j$ pointwise on $\Xqm$, and 
in fact $\f$ is the usc regularization of the pointwise limit of the $\f_j$.

The set of continuous psh metrics can be viewed as the set of \emph{uniform}
limits of positive metrics, in agreement with~\cite{Zha95,GublerLocal}. 
In particular, any metric in $\cH(L)_\R$ is psh.
%
%
%
%
\subsection{Metrics of finite energy}
We extend the Monge--Amp\`ere energy functional to all metrics 
$\f\in\PSH(L)$ by setting 
$E(\f):=\inf\{E(\p)\mid \f\le\p\in\cH(L)\}\in\R\cup\{-\infty\}$ and 
define the class $\cE^1(L)$ as the set of metrics of finite energy, $E(\f)>-\infty$.

The mixed Monge--Amp\`ere operator and the functionals $I$, $J$ 
extend to $\cE^1(L)$ and are continuous under decreasing and increasing 
(but not arbitrary) limits, as is $E$. The inequalities~\eqref{e410} hold on $\cE^1(L)$.
If $\f,\p\in\cE^1(L)$, then $I(\f,\p)=0$ iff $\f-\p$ is a constant, and similarly for $J$ and $I-J$.
%
%
%
%
\subsection{Measures of finite energy and the Calabi--Yau theorem}
The \emph{energy} of a probability measure $\mu$ on $X$ is defined by 
\begin{equation*}
  E^*(\mu)=\sup\{E(\f)-\int\f\,d\mu\mid \f\in\cE(L)\}\in\R\cup\{+\infty\}.
\end{equation*}
This quantity depends
on the ample line bundle $L$, but the set $\cM^1(X)$ of measures of finite energy,
$E^*(\mu)<+\infty$, does not.
The following result will be referred to as the
\emph{Calabi--Yau theorem}. It is proved in~\cite{trivval}, 
see also~\cite{nama,BFJSimons,YZ17},
and is a trivially valued analogue of the fundamental results in~\cite{Yau,GZ2,BBGZ}.
\begin{thm}\cite{trivval}
  The Monge--Amp\`ere operator defines a bijection
  \begin{equation*}
    \MA\colon\cE^1(L)/\R\to\cM^1(X)
  \end{equation*}
  between plurisubharmonic metrics of finite energy modulo constants, 
  and Radon probability measures of finite energy. 
  For any $\f\in\cE^1(L)$, we have $E^*(\MA(\f))=(I-J)(\f)$.
\end{thm}
%
%
%
%
\subsection{Scaling action and homogeneity}
As $k$ is trivially valued, 
there is a \emph{scaling action} of the multiplicative group $\R_+^\times$ on $X$
defined by powers of norms.
We denote by $x^t$ the image of $x$ by $t\in\R_+^\times$.
There is an induced action on functions.
If $\f$ is a metric on a $\Q$-line bundle (viewed as function on $X$)
and $t\in\R_+^\times$, we denote by 
$\f_t$ the metric on $L$ defined by $\f_t(x^t)=t\f(x)$. 
This action preserves the classes $\cH(L)_\R$, $\PSH(L)$, and $\cE^1(L)$. 
The energy functional $E$ is homogeneous in the sense $E(\f_t)=tE(\f)$
for $\f\in\cE^1(L)$ and $t\in\R_+^\times$. The same is true for $I$ and $J$.

There is also an induced action on Radon probability measures.
We have $E^*(t_*\mu)=tE^*(\mu)$ for any $t\in\R_+^\times$,
so the space $\cM^1(X)$ is invariant under scaling. 
Further, $\MA(\f_t)=t_*\MA(\f)$ for $\f\in\cE^1(L)$ and $t\in\R_+^\times$.
%
%
%
%
%
%
\section{Adjoint K-stability and Ding-stability}\label{sec:adjoint}
In~\cite{BermThermo}, Berman introduced ideas from thermodynamics to the 
existence of K\"ahler--Einstein metrics on Fano manifolds. 
In particular, he used the Legendre transform, together
with the (Archimedean) Calabi--Yau theorem to prove that the Mabuchi functional 
is proper on the space of smooth positive metrics on $L$ if and only if the Ding functional
is proper. When $X$ admits no nontrivial vector fields, 
these two conditions are, further, equivalent to the existence of a K\"ahler--Einstein 
metric on $X$~\cite{BBEGZ,DR17}

Here we adapt Berman's ideas to the non-Archimedean setting. Namely, we use the 
Legendre transform and the non-Archimedean Calabi--Yau theorem to prove that 
uniform Ding stability is equivalent to uniform K-stability.
In fact, we work on an arbitrary polarized smooth variety, using adjoint versions of the Ding
and Mabuchi functionals, the latter inducing Dervan's notion of twisted K-stability in the twisted Fano case \cite{Der16}. The results here are used in~\cite{BBJ18},
to give criteria for the existence of twisted K\"ahler--Einstein metrics.

\smallskip
\emph{In the rest of the paper, $X$ is smooth and projective, and 
$L$ is an ample $\Q$-line bundle}.
%
%
%
%
\subsection{Log discrepancy}
There is a natural \emph{log discrepancy} function 
\begin{equation*}
  A=A_X\colon X\to\R_+\cup\{+\infty\}.
\end{equation*}
This is perhaps most naturally viewed as a metric on the canonical 
bundle~\cite{TemkinMetric}, see~\cite{MMS}, but here we present it differently.
First consider a divisorial point $x\in\Xdiv$. There exists $c>0$, 
a proper birational morphism $\cY\to\Xsch$, with $\cY$ smooth, and a
prime divisor $\cD\subset\cY$, such that $x=\exp(-c\,\ord_\cD$).
We then set 
\begin{equation}\label{e407}
  A(x)=c\left(1+\ord_\cD(K_{\cY/\Xsch})\right),
\end{equation}
where $K_{\cY/\Xsch}$ is the relative canonical divisor.

The log discrepancy functional was extended to all of $\Xval$ 
in~\cite{jonmus} (following earlier work in~\cite{valtree,hiro}).
Instead of explaining this construction here, we state the following characterization, 
which is proved in the appendix.
\begin{thm}\label{T403}
  There exists a unique maximal lsc extension $A\colon X\to[0,+\infty]$ 
  of the function $A\colon\Xdiv\to\R_+^\times$ defined above.
  Further, this function satisfies:
  \begin{itemize}
  \item[(i)]
    $A=+\infty$ on $X\setminus\Xval$ and $A<+\infty$ on $\Xqm$;
  \item[(ii)]
    $A(x^t)=tA(x)$ for all $x\in X$ and $t\in\R_+^\times$;
  \item[(iii)]
    for any snc test configuration $\cX$ for $X$, we have
    \begin{itemize}
    \item[(a)]
      $A$ is continuous on the dual complex $\D_\cX$ and affine 
      on each simplex;
    \item[(b)]
      for any point $x\in X$, we have $A(x)\ge A(\retr_\cX(x))$, with equality iff 
      $x\in\D_\cX$;
    \end{itemize}
  \item[(iv)]
    $A=\sup_\cX A\circ\retr_\cX$, where $\cX$ ranges over snc test configurations
    for $X$. 
  \end{itemize}
  Here $\retr_\cX\colon X\to\D_\cX$ denotes the retraction onto the dual complex.
\end{thm}
%
%
%
%
\subsection{Entropy}
The \emph{entropy} of a Radon probability measure $\mu$ on $X$ is defined by
\begin{equation*}
  \Ent(\mu)=\int_XA(x)\,d\mu(x)\in[0,+\infty].
\end{equation*}
where $A$ is the log discrepancy function on $X$.
The integral is well-defined since $A$ is lsc. 

\begin{rmk}
  If $\mu$, $\nu$ are probability measures on a space $X$, then 
  the classical entropy of $\mu$ with respect to $\nu$ is defined as
  $\int\log\left(\frac{d\mu}{d\nu}\right)\,d\mu$,
  when $\mu\ll\nu$, and $+\infty$ otherwise.
  Our notion of entropy can be seen as a non-Archimedean
  degeneration of the usual notion, see~\cite{BHJ2}.
\end{rmk}

The entropy behaves well with respect to regularizations of measures:
\begin{lem}\label{L403}
  For any snc model $\cX$ of $X$, set $\mu_\cX:=(\retr_\cX)_*\mu$. Then 
  $(\Ent(\mu_\cX))_\cX$ forms an increasing net converging to $\Ent(\mu)$.
\end{lem}
\begin{proof}
  Note that $\Ent(\mu_\cX)=\int(A\circ\retr_\cX)\mu$. The result follows since 
  $\mu$ is a Radon measure and $(A\circ\retr_\cX)_\cX$ is an increasing 
  net of lsc functions converging to $A$, see~\cite[7.12]{Folland}.
\end{proof}
The entropy functional is a \emph{linear} functional. In particular,
it is convex. It is homogeneous and lsc on the space of Radon probability measures
(with weak convergence) since $A$ is lsc. However, it is not continuous, since $A$
is not continuous.
%
%
%
%
\subsection{The twisted Mabuchi functional}
Given a $\Q$-line bundle $T$ on $X$, define a relative Monge--Amp\`ere 
energy functional $E_T$ on $\cE^1(L)$ by 
\begin{equation*}
  E_T(\f)
  =(nV)^{-1}\sum_{j=0}^{n-1}\int\f\,\omega_\f^j\wedge\omega^{n-1-j}\wedge\eta.
\end{equation*}
Here $\omega$ and $\eta$ are the curvature forms for the trivial metrics on 
$L$ and $T$, respectively. 
When $\f\in\cH(L)$, $E_T(\f)$ can be computed using intersection numbers,
as in~\cite[\S7.4]{BHJ1}.
That $E_T$ is well-defined and finite
on $\cE^1(L)$ is seen by writing $T$ as a difference between ample
$\Q$-line bundles.
We have $E_T(\f+c)=E_T(\f)+V^{-1}(T\cdot L^{n-1})c$ for $c\in\R$.

As in~\cite{BHJ1}, we define the \emph{Mabuchi} functional $M$ on $\cE^1(L)$ via the 
Chen--Tian formula:
\begin{equation}\label{e432}
  M=H+E_{K_X}+\bar{S}E,
\end{equation}
where $\bar{S}=-nV^{-1}(K_X\cdot L^n)$.
Now, given a $\Q$-line bundle $T$ on $X$, we define 
the \emph{twisted Mabuchi functional} $M_T$ on $\cE^1(L)$ by replacing $K_X$ 
by $K_X+T$ everywhere. This amounts to 
\begin{equation*}
  M_T=M+nE_T-nV^{-1}(T\cdot L^n)E,
\end{equation*}
and is translation invariant: $M_T(\f+c)=M_T(\f)$ for $c\in\R$. 
%
%
%
%
\subsection{The adjoint Mabuchi functional and free energy}
We are interested in the adjoint case $L=-(K_X+T)$, \ie 
$T=-(K_X+L)$, which corresponds to the \emph{twisted Fano case} in the terminology of \cite{Der16}. In this situation, the \emph{adjoint Mabuchi functional}
$\Mad:=M_{-(K_X+L)}$ takes a particularly nice form:
\begin{equation}
  \Mad=H-(I-J),
\end{equation}
where the functional $H\colon\cE^1(L)\to\R_+\cup\{+\infty\}$ is given by
\begin{equation*}
  H(\f)=\Ent(\MA(\f))=\int_XA\MA(\f).
\end{equation*}
While $I$ and $J$ are continuous under decreasing nets, this is not true 
for $H$ and $\Mad$.
\begin{exam}
  Let $X=\P^1$ and $L=\cO(1)$. Pick each $n\ge 1$, pick $2^n$
  divisorial points $x_{n,j}\in X$, $1\le j\le 2^n$ such that $\redu(x_{n,i})\ne\redu(x_{n,j})$
  for $i\ne j$, and $A(x_{n,j})=1$. Define $\f_n\in\PSH(L)$ by $\max(\f_n)=2^{1-n}$ and 
  $\MA(\f_n)=2^{-n}\sum_{j=1}^{2^n}\d_{x_{n,j}}$. Then $\f_n(x_{n,j})=2^{-n}$ for all $n,j$,
  from which it follows that $2^{-n}\le\f_n\le 2^{1-n}$ on $X$. In particular,
  the sequence $(\f_n)_1^\infty$ is decreasing, and converges to 0. 
  Now $H(\f_n)=1$ for all $n$, while $H(0)=0$.
\end{exam}
For this reason---and in contrast to~Lemma~\ref{L408} below---we do not know whether 
$\Mad\ge0$ on $\cH(L)$ implies $\Mad\ge0$ on $\cE^1(L)$.
However, this implication would follow from
\begin{conj}\label{Conj302}
  Given any metric $\f\in\cE^1(L)$, there 
  exists a decreasing net $(\f_j)_j$ of positive metrics converging to $\f$, such that 
  $\lim_j\Ent(\MA(\f_j))=\Ent(\MA(\f))$.
\end{conj}
In the Archimedean case, the corresponding conjecture is true: any metric $\f$ of finite energy
can be approximated by smooth positive metrics, see~\cite[Lemma~3.1]{BDL17}.
The proof in~\loccit~proceeds by approximating the Monge--Amp\`ere measure of $\f$,
and then using the Calabi--Yau theorem. The problem in the non-Archimedean
case is that we don't know how to characterize the image of $\cH(L)$
under the Monge--Amp\`ere operator.
%
%
%
%
\subsection{Free energy}
Being a translation invariant functional on $\cE^1(L)$, the adjoint Mabuchi
functional factors through the Monge--Amp\`ere operator: we have 
\begin{equation*}
  \Mad(\f):=F(\MA(\f)),
\end{equation*}
where the \emph{free energy} functional
$F\colon\cM^1(X)\to\R\cup\{+\infty\}$ is given by 
\begin{equation*}
  F=\Ent-E^*;
\end{equation*}
see~\cite{BermThermo}. In other words,
\begin{equation*}
  F(\mu)=\int_XA\mu-E^*(\mu).
\end{equation*}
Note that while the space $\cM(X)$ does not depend on $L$, the energy functional $E^*$
does; hence the same is true for $F$.
%
%
%
%
\subsection{The Legendre transform of entropy}
Define $L\colon\cE^1(L)\to\R\cup\{-\infty\}$ by
\begin{equation}\label{e402}
  L(\f)=\inf_x(A(x)+\f(x)),
\end{equation}
where the infimum is taken over divisorial valuations $x\in\Xdiv$.
When $X$ is Fano and $L=-K_X$, this extends the functional
in~\cite[Definition~7.26]{BHJ1}.
\begin{prop}\label{P405}
  The infimum in~\eqref{e402} is unchanged when taking it over quasimonomial
  points $x\in\Xqm$, or over points $x\in\Xval$ with $A(x)<\infty$. Further, we have 
  \begin{equation}\label{e403}
    L(\f)=\inf_{\mu\in\cM^1(X)}\{\Ent(\mu)+\int\f\mu\},
  \end{equation}
\end{prop}
\begin{proof}
  Let $L^{\mathrm{qm}}(\f)$ denote the infimum in~\eqref{e402} taken over 
  quasimonomial points, and let $L'(\f)$ denote the right-hand side of~\eqref{e403}.
  We will prove that $L(\f)=L^{\mathrm{qm}}(\f)=L'(\f)$, which implies
  the result since taking $\mu=\d_x$ shows that the infimum in~\eqref{e402} 
  over $x\in\Xval$ with $A(x)<\infty$ is bounded from below and above
  by $L(\f)$ and $L'(\f)$, respectively.

  We have $L^{\mathrm{qm}}(\f)=L(\f)$ since for any snc test configuration
  $\cX$ for $X$, the functions $A$ and $\f$ are continuous on the dual complex
  $\D_\cX$, inside which divisorial points are dense. 

  Taking $\mu=\d_x$ for $x\in\Xqm$ shows that $L'(\f)\le L^{\mathrm{qm}}(\f)$. 
  For the reverse inequality, pick $\e>0$ and
  $\mu\in\cM^1(X)$ such that $\Ent(\mu)+\int\f\,\mu\le L'(\f)+\e$.
  Replacing $\mu$ by $\retr_{\cX*}\mu$ for a large enough $\cX\in\SNC(X)$,
  we may assume $\mu$ is supported on a dual complex $\D_\cX$,
  see Lemma~\ref{L403}. 
  But then it is clear that 
  \begin{equation*}
    \Ent(\mu)+\int\f\,\mu
    =\int_{\D_\cX}(A+\f)\,d\mu
    \ge\inf_{\D_\cX}(A+\f),
  \end{equation*}
  so $L^{\mathrm{qm}}(\f)\le L'(\f)+\e$.
\end{proof}
\begin{lem}\label{L401}
  The functional $L$ is usc and non-increasing on $\cE^1(L)$. 
  As a consequence, it is continuous along decreasing nets in $\cE^1(L)$. 
\end{lem}
\begin{proof}
  The only nontrivial statement is the upper semicontinuity, and this follows
  from the continuity of $\f\mapsto\f(x)$ for $x\in\Xqm$.
\end{proof}
We may think of~\eqref{e403} as saying that $L$ is the \emph{Legendre transform} 
of $\Ent$. Now, the natural setting of the Legendre duality
is between the space $\cM'(X)$ of all signed Radon measures on $X$ 
and the space $C^0(X)$ of continuous functions on $X$. 
Extend $\Ent$ to all of $\cM'(X)$ by $\Ent(\mu)=\int A\mu$ when $\mu$
is a probability measure, and $\Ent(\mu)=+\infty$ otherwise. Then 
define $L(f)$ for $f\in C^0(X)$ by $L(f)=\inf_{x\in X}\{A(x)+f(x)\}$.
\begin{prop}
  For any $f\in C^0(X)$ we have 
  \begin{equation}\label{e405}
    L(f)=\inf_{\mu\in\cM'(X)}\{\Ent(\mu)+\int f\mu\},
  \end{equation}
  and for every $\mu\in\cM'(X)$, we have 
  \begin{equation}\label{e404}
    \Ent(\mu)=\sup_{f\in C^0(X)}\{L(f)-\int f\mu\}.
  \end{equation}
\end{prop}
\begin{proof}
  In~\eqref{e405} it clearly suffices to take the infimum over Radon probability measures 
  $\mu$. The equality then follows as in the proof of Proposition~\ref{P405}.
  
  That~\eqref{e404} holds in now a formal consequence of $A$ being lsc.
  Indeed, fix $\mu\in\cM'(X)$ and let $\Ent'(\mu)$ be the right-hand side of~\eqref{e404}.
  It follows from~\eqref{e405} that $\Ent'(\mu)\le\Ent(\mu)$.
  To prove the reverse inequality, first suppose $\mu$ is a probability measure 
  and pick $\e>0$. Since $-A$ is usc, we can find a continuous function $f$
  on $X$ with $f\ge-A$ and $\int f\,\mu\le-\int A\,\mu+\e=-\Ent(\mu)+\e$.
  Thus $L(f)\ge0$, and hence $\Ent'(\mu)\ge L(f)-\int f\,\mu\ge\Ent(\mu)-\e$.

  Now suppose $\mu\in\cM'(X)$ is not a probability measure. If $\mu(X)\ne1$,
  picking $f\equiv\pm C$, where $C\gg1$ gives $\Ent'(\mu)=+\infty$.
  If $\mu$ is not a positive measure, then there exists $g\in C^0(X)$ with 
  $g\ge0$ but $\int g\,\mu<0$. Picking $f=Cg$ for $C\gg0$ again gives
  $\Ent'(\mu)=+\infty$,
\end{proof}
%
%
%
%
\subsection{The adjoint Ding functional}
Define the \emph{adjoint Ding functional} on $\cE^1(L)$ by 
\begin{equation}
  \Dad=L- E,
\end{equation}
where $L$ is the functional in~\eqref{e402}, and $E$ is the Monge--Amp\`ere energy.
When $X$ is Fano and $L=-K_X$, the restriction of $\Dad$ to $\cH(L)$ coincides
with the Ding functional of~\cite{Berm16,BHJ1}. The name derives from~\cite{Din88}.
\begin{lem}\label{L408}
  The adjoint Ding functional $\Dad$ is continuous along decreasing nets in 
  $\cE^1(L)$. As a consequence, $\Dad\ge0$ on $\cE^1(L)$ iff $\Dad\ge0$ on $\cH(L)$.
\end{lem}
\begin{proof}
  It suffices to prove the continuity assertion, since every element in $\cE^1(L)$ is the limit of 
  a decreasing net in $\cH(L)$. Now, we know that $E$ is continuous along decreasing nets,
  so the result follows from Lemma~\ref{L401}.
\end{proof}
%
%
%
%
\subsection{Non-Archimedean thermodynamics}
Using the Calabi--Yau theorem, we now relate the adjoint Ding and Mabuchi functionals.
The following result can be viewed as a non-Archimedean version of~\cite[Theorem~1.1]{BermThermo}.
\begin{thm}\label{T401}
  Let $L$ be an ample line bundle on a smooth projective variety $X$.
  Then we have $\Mad\ge\Dad$ on $\cE^1(L)$. 
  Further, the following conditions are equivalent:
  \begin{itemize}
  \item[(i)]
    $F\ge0$ on $\cM^1(X)$;
  \item[(ii)]
    $\Mad\ge0$ on $\cE^1(L)$;
  \item[(iii)]
    $\Dad\ge0$ on $\cE^1(L)$.
  \end{itemize}
\end{thm}
In the Archimedean situation, the conditions analogous to~(ii) and~(iii) are that 
the Mabuchi and Ding functionals are bounded from below. In the non-Archimedean
setting, the presence of an $\R_+^\times$-action on $\cE^1(L)$ under which $\Mad$
and $\Dad$ are homogeneous, shows that this is equivalent to~(ii) and~(iii).
\begin{proof}
  Pick any $\f\in\cE^1(L)$. If  $\Dad(\f)>-\infty$, then 
  \begin{equation*}
    \Mad(\f)-\Dad(\f)
    =\int(A+\f)\MA(\f)
    -\inf(A+\f)\ge0,
  \end{equation*}
  so we get $\Mad\ge\Dad$ on $\cE^1(L)$.

  Thus~(iii) implies~(ii). By the Calabi-Yau theorem, it follows 
  that~(ii) implies~(i). Hence it suffices to prove that~(i) implies~(iii).
  By~(i) we have $F(\d_x)\ge0$ for every divisorial point $x\in\Xdiv$, which 
  translates into $A(x)\ge E^*(\d_x)$ for every $x\in\Xdiv$. By the definition
  of $E^*$, this implies $E(\f)-\f(x)\le A(x)$ for every $\f\in\cE^1(L)$.
  Taking the infimum over $x\in\Xdiv$ we get $L(\f)\ge E(\f)$,
  that is, $\Dad(\f)\ge 0$ for all $\f\in\cE^1(L)$.
\end{proof}
%
%
%
%
\subsection{Adjoint semistability}\label{sec:twisted}
We can reformulate Theorem~\ref{T401} as follows.
First, we say that $L$ is \emph{Ding semistable in the adjoint sense} 
if $\Dad\ge0$ on $\cE^1(L)$, and 
\emph{K-semistable in the adjoint sense} if $\Mad\ge0$ on $\cE^1(L)$.

As in \cite[Proposition 8.2]{BHJ1}, one shows that $\Mad\ge 0$ on $\cH(L)$ is equivalent to \emph{twisted K-semistability} in the twisted Fano case, in the sense of \cite{Der16}. According to Conjecture~\ref{Conj302}, $\Mad\ge 0$ on $\cH(L)$ should imply $\Mad\ge 0$ on $\cE^1(L)$, \ie adjoint $K$-semistability. 

Second, we define the \emph{stability threshold} of $L$ as
\begin{equation}\label{e411}
  \d(L)=\inf_{\mu\in\cM(X)}\frac{\Ent(\mu)}{E^*(\mu)}.
\end{equation}
We shall see in Theorem~\ref{T409} that this invariant, 
which was suggested by Berman~\cite{BermComm}, coincides with the one defined 
in~\cite{FO16,BlumJonsson}. In particular, this will show that 
$\d(L)>0$.
\begin{cor}\label{C407}
    For any ample line bundle $L$ on $X$, the following are equivalent:
  \begin{itemize}
  \item[(i)]
    $\d(L)\ge1$;
  \item[(ii)]
    $L$ is Ding semistable in the adjoint sense;
  \item[(ii)]
    $L$ is K-semistable in the adjoint sense.
  \end{itemize}
\end{cor}
%
%
%
%
\subsection{Uniform adjoint stability}
In analogy with~\cite{BHJ1,Der16}) we say that $L$ is 
\emph{uniformly Ding-stable in the adjoint sense}
if there exists $\e>0$ such that $\Dad\ge\e J$ on $\cE^1(L)$.
Similarly, $L$ is \emph{uniformly K-stable in the adjoint sense}
if there exists $\e>0$ such that $\Mad\ge\e J$ on $\cE^1(L)$. Here again, this is equivalent (at least up to Conjecture~\ref{Conj302}) to uniform twisted K-stability in the twisted Fano case, in the sense of \cite{Der16}. 

\begin{thm}\label{T405}
    For any ample line bundle $L$, the following are equivalent:
  \begin{itemize}
  \item[(i)]
    $\d(L)>1$;
    \item[(ii)]
    $L$ is uniformly K-stable in the adjoint sense;
  \item[(iii)]
    $L$ is uniformly Ding-stable in the adjoint sense.
\end{itemize}
\end{thm}
Together with Corollary~\ref{C407}, this proves Theorem~A in the introduction.
\begin{proof}
  The Calabi--Yau theorem shows that if $\d\ge1$, then 
  $\Ent\ge\d E^*$ on $\cM^1(X)$ iff $\Mad\ge(\d-1)(I-J)$
  on $\cE^1(X)$. Since the functionals $I-J$ and $J$ are comparable,
  see~\eqref{e410}, this implies that~(i)$\Leftrightarrow$(ii). 
  The inequality $\Mad\ge\Dad$ shows that~(iii) implies~(ii).

  It remains to prove that~(i) implies~(iii).  We follow the proof of Theorem~\ref{T401}.
  Pick any $\d\in(1,\d(L))$. 
  Then $\Ent(\mu)\ge\d E^*(\mu)$ for any $\mu\in\cM^1(X)$.
  When $\mu=\d_x$, this gives $A(x)\ge \d E^*(\d_x)$
  for any $x\in\Xdiv$. Now consider any $\f\in\cE^1(L)$. 
  Since $\d>1$ and $\cE^1(L)$ is convex, we have $\d^{-1}\f\in\cE^1(L)$. 
  This gives $E^*(\d_x)\ge E(\d^{-1}\f)-\d^{-1}\f(x)$.
  Thus $A(x)+\f(x)\ge\d E^*(\d_x)+\f(x)\ge\d E(\d^{-1}\f)$. Taking the infimum
  over $x\in\Xdiv$ and subtracting $E(\f)$ gives
  $\Dad(\f)\ge \d E(\d^{-1}\f)-E(\f)$.
  By translation invariance, we may assume $\sup_X\f=0$. Then $E(\f)=-J(\f)$
  and $E(\d^{-1}\f)=-J(\d^{-1}\f)$. By~\cite[Lemma~6.17]{trivval}, we have 
  \begin{equation*}
    J(\d^{-1}\f)\le \d^{-(1+n^{-1})}J(\f).
  \end{equation*}
  This implies $\Dad(\f)\ge\e J(\f)$, with $\e=(1-\d^{-n^{-1}})$.
\end{proof}
%
%
%
%
\subsection{The Fano case}
The adjoint stability notions above are defined in terms
of the space $\cE^1(L)$ of metrics of finite energy. This is a natural framework for 
applying the Calabi--Yau theorem. On the other hand, 
K-stability and Ding stability are usually expressed in terms of test configurations, that is,
metrics in $\cH(L)$. For Ding-stability, this makes no difference, in view of
Lemma~\ref{L408}.
If Conjecture~\ref{Conj302} holds, then the same is true for K-stability. 

When $X$ is Fano and $L=-K_X$, the adjoint notions do coincide with
the usual ones:
\begin{thm}\label{T402}
  If $X$ is a Fano manifold, then the following are equivalent:
  \begin{itemize}
  \item[(i)]
    $X$ is $K$-semistable (resp.\ uniformly $K$-stable);
  \item[(ii)]
    $X$ is $K$-semistable (resp.\ uniformly $K$-stable) in the adjoint sense;
  \item[(iii)]
    $X$ is Ding semistable (resp.\ uniformly Ding-stable);
  \item[(iv)]
    $X$ is Ding semistable (resp.\ uniformly Ding-stable) in the adjoint sense;
  \item[(v)]
    $\d(-K_X)\ge1$ (resp.\ $\d(-K_X)>1$.
  \end{itemize}
\end{thm}
\begin{proof}
  The equivalence of~(ii),~(iv) and~(v) follow from Corollary~\ref{C407} and
  Theorem~\ref{T405}. Now Theorem~\ref{T409} below shows that $\d(-K_X)$
  agrees with the invariant considered in~\cite{FO16,BlumJonsson}.
  The equivalence of~(i),~(iii) and~(v) therefore follow from~\cite[Theorem~B]{BlumJonsson}.
\end{proof}
%
%
%
%
%
\section{Graded norms and filtrations}\label{s:filtnorm}
In this section, we study the space of bounded 
graded norms (or, equivalently, filtrations) on the section ring $R=R(X,L)$.
As before, $k$ is a trivially valued field, whereas
$L$ is an ample line bundle (as opposed to a $\Q$-line bundle) on $X$.

Much of the material here is studied for more general
valued fields $k$ in~\cite{Chen-Maclean,BE18}, but we present the details for the 
convenience of the reader.
There is also some overlap with the recent work of Codogni~\cite{Cod18}.
%
%
%
%
\subsection{Norms and filtrations}
Let $V$ be a $k$-vector space.
By a \emph{filtration} $\cF$ of $V$ we mean
a family $(\cF^\la V)_{\la\in\R}$ of $k$-vector subspaces of $V$, 
satisfying $\cF^\la V=\bigcap_{\la'<\la}\cF^{\la'}V$,
$\bigcup_\la\cF^\la V=V$, and $\bigcap_\la\cF^\la V=0$.
A filtration is \emph{bounded} if $\cF^\la V=V$ for $\la\ll0$
and $\cF^\la V=0$ for $\la\gg0$.

Filtrations of $V$ are in bijection with 
(non-Archimedean) \emph{norms} on $V$, \ie functions
$\|\cdot\|\colon V\to\R_+$ satisfying $\|v\|=0$ iff $v=0$, 
$\|av\|=|a|\|v\|$ for $a\in k$, $v\in V$, and 
$\|v+w\|\le\max\{\|v\|,\|w\|\}$ for $v,w\in V$.
To a filtration $\cF$ is associated the 
norm $\|v\|=\exp(-\sup\{\la\in\R\mid v\in\cF^\la V\})$;
conversely, a norm $\|\cdot\|$ on $V$ induces the filtration
$\cF^\la V=\{v\in V\mid \|v\|\le\exp(-\la)\}$.
In what follows, we will usually work with norms rather than filtrations.

Bounded filtrations correspond to bounded norms, \ie norms for
which there exists $A>0$ such that 
$A^{-1}\le\|v\|\le A$ for all $v\ne0$. If $V$ is finite dimensional,
then any filtration/norm on $V$ is bounded.
The \emph{trivial} norm on $V$ is defined by $\|v\|=1$ for $v\ne0$.
A norm is \emph{almost trivial} if it is a multiple of the trivial norm.

If $V$ is a normed vector space, any subspace $W\subset V$ is naturally
equipped with the subspace norm, and the quotient $V/W$ 
with the quotient norm defined by $\|v+W\|:=\inf\{\|v+w\|\mid w\in W\}$.
In general, this is only a seminorm on $V/W$ (\ie there may be nonzero
elements of norm zero) but it is a norm when $V$ is finite dimensional.

The space $\cN_V$ of norms on $V$ admits two natural operations.
First, if $\|\cdot\|$ and $\|\cdot\|'$ are norms on $V$,
so is their maximum $\|\cdot\|\vee\|\cdot\|':=\max\{\|\cdot\|,\|\cdot\|'\}$.
Second, we have an action of $\R$ on $\cN_V$ given by 
$(t,\|\cdot\|)\mapsto\exp(t)\|\cdot\|$.
%
%
%
%
\subsection{Relative successive minima and volume}
As $k$ is trivially valued, any finite-dimensional
normed $k$-vector space $V$ admits a basis $\{e_j\}_j$ that is \emph{orthogonal}
for the norm in the sense that $\|\sum_ja_je_j\|=\max_j|a_j|\|e_j\|$
for all $a_j\in k$.
More generally, given any two norms $\|\cdot\|$, $\|\cdot\|'$ 
on a finite-dimensional $k$-vector space $V$,
there exists a basis $\{e_j\}_{j=1}^N$ for $V$ that is orthogonal for both norms.
The numbers 
\begin{equation*}
  \la_j:=\log\frac{\|e_j\|'}{\|e_j\|},\quad 1\le j\le N,
\end{equation*}
are called the 
\emph{relative successive minima} of $\|\cdot\|$ with respect to $\|\cdot\|'$.
They do not depend on the choice of orthogonal basis\footnote{This is not completely obvious: see~\cite[\S3.1]{BE18}}.

Following~\cite{BE18}, the average of the relative successive minima,
\begin{equation}\label{e417}
  \vol(\|\cdot\|,\|\cdot\|'):=N^{-1}\sum_{j=1}^N\la_j,
\end{equation}
is called the (logarithmic) \emph{relative volume}
of $\|\cdot\|$ with respect to $\|\cdot\|'$. It can be described as follows.
The norms $\|\cdot\|$, $\|\cdot\|'$ on $V$ canonically induce norms 
$\det\!\|\cdot\|$, $\det\!\|\cdot\|'$ on the determinant line $\det V$, and 
\begin{equation}\label{e422}
  \vol(\|\cdot\|,\|\cdot\|')
  =N^{-1}(\log\det\!\|\eta\|'-\log\det\!\|\eta\|)
\end{equation}
for any nonzero element $\eta\in\det V$. As a consequence we have the 
\emph{cocycle condition}
\begin{equation}\label{eq:cocycle}
  \vol(\|\cdot\|,\|\cdot\|')+\vol(\|\cdot\|',\|\cdot\|'')=\vol(\|\cdot\|,\|\cdot\|'')
\end{equation}
for any three norms on $V$.

When $\|\cdot\|'$ is the trivial norm, we drop the term ``relative'' and
simply say \emph{successive minima}
and \emph{volume}, and write $\vol(\|\cdot\|)$ for the latter.
The successive minima of a norm are exactly the \emph{jumping numbers} of the 
associated filtration, \ie the $\la\in\R$ such that $\cF^\la\supsetneq\cF^{\la'}$ for any $\la'>\la$ (counted with multiplicity).
%
%
%
%
\subsection{Distances}
We use the relative successive minima to define a distances $d_p$, $1\le p\le\infty$
on the space $\cN_V$ of norms on $V$. Namely, we set 
\begin{equation}\label{e418}
  d_p(\|\cdot\|,\|\cdot\|'):=(N^{-1}\sum_{j=1}^N|\la_j|^p)^{1/p}
\end{equation}
for $p\in[1,\infty)$, and 
$d_\infty(\|\cdot\|,\|\cdot\|'):=\max_{1\le j\le N}|\la_j|$.
One can prove that $d_p$ satisfies the triangle inequality, see~\cite[\S3.1]{BE18}. 
Note that 
\begin{equation}\label{e420}
  d_p(\|\cdot\|,\|\cdot\|')^p
  =d_p(\|\cdot\|,\|\cdot\|\vee\|\cdot\|')^p
  +d_p(\|\cdot\|',\|\cdot\|\vee\|\cdot\|')^p
\end{equation}
for $p\in[1,\infty)$.
There is a similar formula when $p=\infty$.

The distance $d_1$ is easier to control than the others, because of its close relationship
to the relative volume.
Indeed, if $\|\cdot\|\le\|\cdot\|'$ pointwise on $V$, then 
\begin{equation}\label{e419}
  d_1(\|\cdot\|,\|\cdot\|')=\vol(\|\cdot\|,\|\cdot\|').
\end{equation}
We will later need the following estimate.
\begin{lem}\label{L407}
  Let $\|\cdot\|_i$, $\|\cdot\|'_i$, $i=1,2$, be norms on $V$. Then 
  \begin{equation}\label{e421}
    d_1(\|\cdot\|_1\vee\|\cdot\|_2,\|\cdot\|'_1\vee\|\cdot\|'_2)
    \le d_1(\|\cdot\|_1,\|\cdot\|'_1)+d_1(\|\cdot\|_2,\|\cdot\|'_2).
  \end{equation}
\end{lem}
\begin{proof}
  To begin, consider the case when $\|\cdot\|_i\le\|\cdot\|'_i$, $i=1,2$. 

  Further, we first assume there exists a basis $e=(e_1,\dots,e_N)$ for $V$ 
  that is orthogonal for
  all four norms. Write $\|e_j\|_i=\exp(a_{i,j})$ and $\|e_j\|'_i=\exp(a'_{i,j})$ for
  $1\le j\le N$ and $i=1,2$. Then $a_{i,j}\le a'_{i,j}$ for all $i,j$, and
  we must show that
  \begin{equation*}
    \sum_{j=1}^Na'_{1,j}\vee a'_{2,j}-\sum_{j=1}^Na_{1,j}\vee a_{2,j}
    \le
    \sum_{j=1}^N(a'_{1,j}-a_{1,j})+\sum_{j=1}^N(a'_{2,j}-a_{2,j});
  \end{equation*}
  this is straightforward.

  When no such basis exists, we 
  use the following construction: see~\cite[\S3.1]{BE18} for details. 
  For any basis $e=(e_1,\dots,e_N)$ of $V$ there is a 
  ``projection'' $\rho_e\colon\cN_V\to\cN_V$ with the following properties:
  (1) $\rho_e(\|\cdot\|)=\|\cdot\|$ iff $\|\cdot\|$ is orthogonal for $e$;
  (2) $\rho_e\circ\rho_e=\rho_e$; 
  (3) $\det\!\rho_e(\|\cdot\|)=\det\!\|\cdot\|$; and
  (4) if $\|\cdot\|\le\|\cdot\|'$, then $\rho_e(\|\cdot\|)\le\rho_e(\|\cdot\|')$.

  Now assume $e$ is orthogonal for $\|\cdot\|'_i$, $i=1,2$.
  Replacing $\|\cdot\|_i$ by $\rho_e(\|\cdot\|_i)$, $i=1,2$ does not change the 
  right-hand side of~\eqref{e421}. As for the left-hand side,~(2) and~(4) above imply
  \begin{equation*}
    \rho_e(\|\cdot\|_1)\vee\rho_e(\|\cdot\|_2)
    \le\rho_e(\|\cdot\|_1\vee\|\cdot\|_2)
    \le\|\cdot\|'_1\vee\|\cdot\|'_2,
  \end{equation*}
  which in view of~(3) and~\eqref{e422} implies that the left-hand side of~\eqref{e421}
  can only increase upon replacing $\|\cdot\|_i$ by $\rho_e(\|\cdot\|_i)$, $i=1,2$.

  \smallskip
  Finally consider arbitrary norms. Set $\|\cdot\|''_i=\|\cdot\|_i\vee\|\cdot\|'_i$
  for $i=1,2$. By~\eqref{e420} we have 
  \begin{equation*}
    d_1(\|\cdot\|_1\vee\|\cdot\|_2,\|\cdot\|'_1\vee\|\cdot\|'_2)
    =d_1(\|\cdot\|_1\vee\|\cdot\|_2,\|\cdot\|''_1\vee\|\cdot\|''_2)
    +d_1(\|\cdot\|'_1\vee\|\cdot\|'_2,\|\cdot\|''_1\vee\|\cdot\|''_2)
  \end{equation*}
  and $\|\cdot\|_i,\|\cdot\|'_i\le\|\cdot\|''_i$, for $i=1,2$,
  so~\eqref{e421} follows from~\eqref{e420} and the case just considered.
\end{proof}
%
%
%
%
\subsection{Graded norms and filtrations}
Let $L$ be an ample line bundle.
For $m\in\N$, write $R_m:=H^0(X,mL)$. Thus $R_m\ne0$ for $m\gg0$.
Consider the section ring $R=\bigoplus_mR_m$. 

A \emph{graded norm} $\|\cdot\|_\bullet$ 
on $R$ is the data of a norm $\|\cdot\|_m$ on the $k$-vector space $R_m$ for each $m$,
satisfying $\|s\otimes s'\|_{m+m'}\le\|s\|_m\cdot\|s'\|_{m'}$ for $s\in R_m$, $s'\in R_{m'}$.
Since $R$ is finitely generated, there exists $C\ge0$ such that 
$\|\cdot\|_m\le\exp(Cm)$ on $R_m$ for all $m$. We say
$\|\cdot\|$ is (exponentially) \emph{bounded} if 
$\|\cdot\|_m\ge\exp(-Cm)$ on $R_m\setminus\{0\}$ for some $C\ge0$ and all $m$. 

If $\|\cdot\|_\bullet$ and $\|\cdot\|'_\bullet$ are (bounded) graded norms on $V$, so is
their maximum $\|\cdot\|_\bullet\vee\|\cdot\|'_\bullet$.
If $\|\cdot\|_\bullet$ is a (bounded) graded norm and $c\in\R$, then $\exp(c\,\bullet)\|\cdot\|_\bullet$, 
defined by $\exp(cm)\|\cdot\|_m$ on $R_m$, is a (bounded) graded norm on $V$.
The \emph{trivial graded norm} on $R$ is the graded norm
for which $\|\cdot\|_m$ is the trivial norm on $R_m$ for every $m$.

A graded norm is \emph{generated in degree one} if $R$ is generated in degree
one, that is, the canonical morphism $S^mR_1\to R_m$ is surjective for all $m\ge1$,
and the associated norm on $R_m$ is equal to the quotient norm from this morphism.
If $R$ is generated in degree 1, and $\|\cdot\|_1$ is any norm
on $R_1$, then $R$ admits a unique graded norm that is generated in degree one
and extends $\|\cdot\|_1$.
A graded norm is \emph{finitely generated} if the induced graded norm 
on $R(X,rL)$ is generated in degree 1 for some $r\ge1$.

A \emph{graded filtration} on $R$ is the collection of a filtration $(\cF^\la R_m)_\la$
of $R_m$ for all $m$, satisfying $\cF^\la R_m\cdot\cF^{\la'}R_{m'}\subset\cF^{\la+\la'}R_{m+m'}$.
As above, graded norms on $R$ are in bijection with graded filtrations of $R$,
and bounded graded norms correspond to (linearly) bounded
graded filtrations, \ie graded filtrations for which there exists $C\in\R$ such that 
$\cF^\la R_m=0$ for $\la\ge Cm$ and $\cF^\la R_m=R_m$ for $\la\le-Cm$.
We say that a graded filtration is generated in degree one if the associated graded norm
is generated in degree one.
The trivial graded norm on $R$ corresponds to the trivial graded filtration of $R$, defined by 
$\cF^\la R_m=R_m$ for $\la\le 0$ and $\cF^\la R_m=0$ for $\la>0$.

A graded filtration $\cF$ of $R(X,L)$ is a graded \emph{$\Z$-filtration} if all jumping numbers are integers, \ie
$\cF^\la R_m=\cF^{\lceil \la \rceil}R_m$ for all $\la$ and $m$. 
They correspond to graded norms taking values in $\{0\}\cup\exp(\Z)$.
Any graded filtration $\cF$  induces a graded $\Z$-filtration $\cF_\Z$ by setting
$\cF_\Z^\la R_m:= \cF^{\lceil \la \rceil} R_m$.
There is a similar operation on graded norms.

Finitely generated bounded $\Z$-filtrations of $R(X,L)$ are in 1-1 correspondence with 
ample test configurations for $(X,L)$, see~\cite[\S2.5]{BHJ1}.
%
%
%
%
\subsection{Relative limit measures}
Let $\|\cdot\|_\bullet$ and $\|\cdot\|'_\bullet$ be bounded graded norms on $R(X,L)$.
For $m\ge1$, let $\la_{m,j}$, $1\le j\le N_m$ be the relative successive
minima of $\|\cdot\|_m$ with respect to $\|\cdot\|'_m$.
Since the graded norms are bounded, there exists $C>0$ such that 
$|\la_{m,j}|\le Cm$ for all $m,j$. The following result was proved
by Chen and Maclean~\cite{Chen-Maclean}, building upon~\cite{BC11}.
\begin{thm}\label{T408}
  There exists a compactly supported Borel probability
  measure $\nu$ on $\R$ such that the probability measures
  \begin{equation*}
    \nu_m:=\frac1{N_m}\sum_{j=1}^{N_m}\delta_{\la_{m,j}/m}
  \end{equation*}
  converge weakly to $\nu$ as $m\to\infty$.
\end{thm}
We call $\nu$ is the \emph{relative limit measure}
of $\|\cdot\|_\bullet$ with respect to $\|\cdot\|'_\bullet$. To indicate
the dependence on the graded norms, we write 
$\nu=\RLM(\|\cdot\|_\bullet,\|\cdot\|'_\bullet)$.
When $\|\cdot\|'_\bullet$ is the trivial graded norm, we write $\nu=\LM(\|\cdot\|_\bullet)$
and call it the \emph{limit measure} of $\|\cdot\|_\bullet$ following~\cite{BC11}.
\begin{proof}
  Since our setting and notation differs slightly from~\cite{Chen-Maclean},
  we sketch the proof.
  The idea is to reduce to the case when $\|\cdot\|'_\bullet$ is the trivial graded norm;
  this case was treated in~\cite{BC11} (see also~\cite{Bou14}) using the technique
  of Okounkov bodies. 

  Since $L$ is ample, $R_m\ne0$ for $m\gg0$.
  As already noted, there exists $C>0$ such that $\nu_m$ is supported in $[-C,C]$
  for $m\gg0$. 
  It suffices to prove that 
  \begin{equation*}
    \int_\R\max\{\la,c\}\,d\nu_m(\la)
    =\frac1{mN_m}\sum_{j=1}^{N_m}\max\{\la_{j,m},mc\}
  \end{equation*}
  converges as $m\to\infty$, for all $c\in\R$,
  see~\cite[Proposition~5.1]{Chen-Maclean}.
  But the numbers $\max\{\la_{j,m},cm\}$, $1\le j\le N_m$, are the 
  relative successive minima of $\|\cdot\|_m$ with respect to the norm
  $\|\cdot\|'_m\vee\exp(cm)\|\cdot\|_m$ on $R_m$.
  Replacing $\|\cdot\|'_\bullet$ by $\|\cdot\|'_\bullet\vee\exp(c\,\bullet)\|\cdot\|_\bullet$,
  we are reduced to proving that the \emph{barycenters} $(mN_m)^{-1}\sum_j\la_{m,j}$
  of the measures $\nu_m$ converge as $m\to\infty$.
  But if $\{e_{m,j}\}_j$ is a basis for $R_m$ that is orthogonal for both norms,  then
  \begin{equation*}
    \frac{1}{mN_m}\sum_{j=1}^{N_m}\la_{m,j}
    =\frac{1}{mN_m}\sum_{j=1}^{N_m}\log\frac{\|e_{m,j}\|'}{\|e_{m,j}\|}
    =\frac{1}{mN_m}\sum_{j=1}^{N_m}\log\frac{1}{\|e_{m,j}\|}
    -\frac{1}{mN_m}\sum_{j=1}^{N_m}\log\frac{1}{\|e_{m,j}\|'}
  \end{equation*}
  is the difference of the barycenters of the probability measures 
  defined by the successive minima of $\|\cdot\|_m$ and $\|\cdot\|'_m$
  respectively, and hence converges to the difference of the barycenters of the 
  limit measures of $\|\cdot\|_\bullet$ and $\|\cdot\|'_\bullet$.
\end{proof}
\begin{cor}\label{C406}
  For any two bounded graded norms  $\|\cdot\|$ and $\|\cdot\|'$ on $R$,
  the limit 
  \begin{equation*}
    \vol(\|\cdot\|_\bullet,\|\cdot\|'_\bullet)
    :=\lim_{m\to\infty}m^{-1}\vol(\|\cdot\|_m,\|\cdot\|'_m)
  \end{equation*}
  exists, and equals the barycenter of the relative limit measure
  $\RLM(\|\cdot\|_\bullet,\|\cdot\|'_\bullet)$.
\end{cor}
We call $\vol(\|\cdot\|_\bullet,\|\cdot\|'_\bullet)$ the \emph{relative volume}
of $\|\cdot\|_\bullet$ with respect to $\|\cdot\|'_\bullet$. It satisfies
a cocycle condition as in~\eqref{eq:cocycle}. The proof of Theorem~\ref{T408}
shows that if $\nu=\RLM(\|\cdot\|_\bullet,\|\cdot\|'_\bullet)$, then 
\begin{equation}\label{e423}
  \int\max\{\la,c\}\,d\nu(\la)
  =\vol(\|\cdot\|_\bullet,\|\cdot\|'_\bullet\vee\exp(c\,\bullet)\|\cdot\|_\bullet).
\end{equation}

\begin{proof}[Proof of Corollary~\ref{C406}]
  This follows from Theorem~\ref{T408} since 
  $m^{-1}\vol(\|\cdot\|_m,\|\cdot\|'_m)$ is the barycenter of the measure $\nu_m$,
  which converges weakly to $\nu:=\RLM(\|\cdot\|_\bullet,\|\cdot\|'_\bullet)$. Indeed, 
  the supports of all the $\nu_m$ are all contained in a fixed interval $[-C,C]$.
\end{proof}
The next results show how relative limit measures behave under
operations on graded norms. 
They follow from elementary computations of relative successive minima
in bases for $R_m$ that are orthogonal for both $\|\cdot\|_m$ and $\|\cdot\|'_m$. 
The details are left to the reader.
\begin{prop}\label{P401}
  Let $\|\cdot\|_\bullet$ and $\|\cdot\|'_\bullet$ be bounded graded norms on $R$,
  and $c\in\R$. Then:
  \begin{itemize}
  \item[(i)]
    $\RLM(\|\cdot\|'_\bullet,\|\cdot\|_\bullet)$ is the pushforward of 
    $\RLM(\|\cdot\|_\bullet,\|\cdot\|'_\bullet)$ 
    under $\la\mapsto-\la$.
  \item[(ii)]
    $\RLM(\exp(c\,\bullet)\|\cdot\|_\bullet,\|\cdot\|'_\bullet)$ is the pushforward of
    $\RLM(\|\cdot\|_\bullet,\|\cdot\|'_\bullet)$ under $\la\mapsto\la-c$.
  \item[(iii)]
    $\RLM(\|\cdot\|_\bullet,\|\cdot\|_\bullet\vee\|\cdot\|'_\bullet)$ is the 
pushforward 
    of $\RLM(\|\cdot\|_\bullet,\|\cdot\|'_\bullet)$ under $\la\mapsto\max\{\la,0\}$.
 \end{itemize}
\end{prop}
\begin{prop}\label{P407}
  Let $\|\cdot\|_\bullet$ and $\|\cdot\|'_\bullet$ be bounded graded norms on $R$.
  Given $r\ge1$, let $\|\cdot\|^{(r)}_\bullet$ and $\|\cdot\|^{\prime(r)}_\bullet$ be 
  their restrictions to $R(X,rL)$. Then 
  $\RLM(\|\cdot\|^{(r)}_\bullet,\|\cdot\|^{\prime(r)}_\bullet)$ is the
  pushforward 
  of $\RLM(\|\cdot\|_\bullet,\|\cdot\|'_\bullet)$ under $\la\mapsto r\la$.
\end{prop}
%
%
%
%
\subsection{Equivalence of bounded graded norms}
Theorem~\ref{T408} yields
\begin{cor}\label{C403}
  Let $\|\cdot\|_\bullet$ and $\|\cdot\|'_\bullet$ be bounded graded norms on $R$.
  For $p\in[1,\infty)$, the limit 
  \begin{equation*}
    d_p(\|\cdot\|_\bullet,\|\cdot\|'_\bullet)=\lim m^{-1}d_{m,p}(\|\cdot\|_m,\|\cdot\|'_m)
  \end{equation*}
  exists, and we have
  \begin{equation}\label{e401}
    d_p(\|\cdot\|_\bullet,\|\cdot\|'_\bullet)=\int_\R|\la|^p\,d\nu(\la),
  \end{equation}
  where $\nu=\RLM(\|\cdot\|_\bullet,\|\cdot\|'_\bullet)$ is the relative limit measure of 
  $\|\cdot\|_\bullet$ with respect to $\|\cdot\|'_\bullet$.
\end{cor}
The function $d_p$ on pairs of bounded graded norms is symmetric and satisfies
the triangle inequality. Further~\eqref{e420} implies
\begin{equation}\label{e415}
  d_p(\|\cdot\|_\bullet,\|\cdot\|'_\bullet)^p
  =d_p(\|\cdot\|_\bullet,\|\cdot\|_\bullet\vee\|\cdot\|'_\bullet)^p
  +d_p(\|\cdot\|'_\bullet,\|\cdot\|_\bullet\vee\|\cdot\|'_\bullet)^p.
\end{equation}
In particular, the distance $d_1$ can be computed in terms of relative volumes:
\begin{equation}\label{e426}
  d_1(\|\cdot\|_\bullet,\|\cdot\|'_\bullet)
  =\vol(\|\cdot\|_\bullet,\|\cdot\|_\bullet\vee\|\cdot\|'_\bullet)
  +\vol(\|\cdot\|'_\bullet,\|\cdot\|_\bullet\vee\|\cdot\|'_\bullet).
\end{equation}

We'd like to analyze how far $d_p$ is from being  a distance.
 If $\|\cdot\|_\bullet$ and $\|\cdot\|'_\bullet$ are 
bounded graded norms on $R$, then 
$p\mapsto d_p(\|\cdot\|_\bullet,\|\cdot\|'_\bullet)$ is increasing.
Further, if $C\ge0$ is such that 
\begin{equation*}
  \exp(-Cm)\le\|\cdot\|_m,\|\cdot\|'_m\le\exp(Cm)
\end{equation*}
on $R_m$ for all $m\ge1$, then
$d_p(\|\cdot\|_\bullet,\|\cdot\|'_\bullet)\le C^{p-1}d_1(\|\cdot\|_\bullet,\|\cdot\|'_\bullet)$
for all $p\in[1,\infty)$. 
In particular, for any two graded norms and any $p\in[1,\infty)$, we have 
$d_p(\|\cdot\|_\bullet,\|\cdot\|'_\bullet)=0$ iff 
$d_1(\|\cdot\|_\bullet,\|\cdot\|'_\bullet)=0$.
In this case we say that the two graded norms are \emph{equivalent}.
As a consequence of Corollary~\ref{C403} we have:
\begin{cor}\label{C402}
  If $\|\cdot\|_\bullet$ and $\|\cdot\|'_\bullet$ are bounded graded norms on $R$, then 
  $\|\cdot\|_\bullet$ and $\|\cdot\|'_\bullet$ are equivalent iff 
  the relative limit measure $\RLM(\|\cdot\|_\bullet,\|\cdot\|'_\bullet)$ is 
  a Dirac mass at the origin.
\end{cor}
\begin{rmk}
  The equivalence notion above is the natural one in our study, but there other possibilities.
  For example, let us say two bounded graded norms $\|\cdot\|_\bullet$ and $\|\cdot\|'_\bullet$
  are \emph{$\infty$-equivalent} if $\lim m^{-1}d_\infty(\|\cdot\|_m,\|\cdot\|'_m)=0$.
  This implies equivalence in the sense above, but the converse is not true. 
  Indeed, let $X=\P^1$ and $L=\cO(1)$, so that $R=k[S,T]$. Let $\|\cdot\|'_\bullet$ be the trivial
  graded norm, and define a graded norm $\|\cdot\|_\bullet$ 
  on $R$ by $\|Q(S,T)\|_m=\exp(m)$ if $S\nmid Q$ and $\|Q(S,T)\|_m=1$ if 
  $S\mid Q$. The measure in Theorem~\ref{T408} is given by 
  $\nu_m=\frac{1}{m+1}\delta_{-1}+\frac{m}{m+1}\delta_0$. 
  Thus the relative limit measure is equal to $\d_0$, but 
  $d_\infty(\|\cdot\|_m,|\cdot\|'_m)=m$ for all $m$.
  See also~\cite[p.458]{Sze15}.
\end{rmk}
All the constructions so far involving graded norms only depend on equivalence classes:
\begin{thm}\label{T411}
  Let $\|\cdot\|_{i,\bullet}$ and $\|\cdot\|'_{i,\bullet}$, $i=1,2$, be bounded graded norms
  on $R$ such that $\|\cdot\|_{i,\bullet}$ is equivalent to $\|\cdot\|'_{i,\bullet}$. Then
  $\RLM(\|\cdot\|_{1,\bullet},\|\cdot\|_{2,\bullet})=\RLM(\|\cdot\|'_{1,\bullet},\|\cdot\|'_{2,\bullet})$.
  As a consequence, 
  $\vol(\|\cdot\|_{1,\bullet},\|\cdot\|_{2,\bullet})=\vol(\|\cdot\|'_{1,\bullet},\|\cdot\|'_{2,\bullet})$.
\end{thm}
\begin{lem}\label{L411}
  With notation and assumptions as in Theorem~\ref{T411}, we have:
  \begin{itemize}
  \item[(i)]
    $\|\cdot\|_{1,\bullet}\vee\|\cdot\|_{2,\bullet}$ is equivalent to
    $\|\cdot\|'_{1,\bullet}\vee\|\cdot\|'_{2,\bullet}$;
  \item[(ii)]
    For $i=1,2$ and any $c\in\R$, 
    $\exp(c\,\bullet)\|\cdot\|_{i,\bullet}$ and $\exp(c\,\bullet)\|\cdot\|'_{i,\bullet}$
    are equivalent.
  \end{itemize}
\end{lem}
\begin{proof}
  The estimate in Lemma~\ref{L407} implies a corresponding estimate for
  bounded graded norms, and yields~(i).
  The statement in~(ii) follows from the equality
  \begin{equation*}
    d_1(\exp(c\,\bullet)\|\cdot\|_\bullet,\exp(c\,\bullet)\|\cdot\|'_\bullet)
    =d_1(\|\cdot\|_\bullet,\|\cdot\|'_\bullet)
  \end{equation*}
  for any two bounded graded norms.
\end{proof}
\begin{proof}[Proof of Theorem~\ref{T411}]
  Set $\nu:=\RLM(\|\cdot\|_{1,\bullet},\|\cdot\|_{2,\bullet})$ 
  and $\nu':=\RLM(\|\cdot\|'_{1,\bullet},\|\cdot\|'_{2,\bullet})$.
  We must show that $\nu=\nu'$.
  It suffices to prove that 
  $\int\max\{\la,c\}\,d\nu(\la)=\int\max\{\la,c\}\,d\nu'(\la)$
  for all $c\in\R$. This follows from~\eqref{e423} and Lemma~\ref{L411}.
\end{proof}
%
%
%
%
\subsection{Almost trivial graded norms and filtrations}\label{S407}
Given $p\in[1,\infty)$ we define the \emph{$p$th central moment} $\|\cF\|_p$ 
of a graded filtration $\cF$ of $R$ as the $p$th central moment of its limit measure.
Thus $\|\cF\|_p=(\int|\la-\overline\la|^p\,d\nu(\la))^{1/p}$, where 
$\nu$ is the limit measure and $\overline\la=\int\la\,d\nu(\la)$ is its barycenter.
We define $p$th central moments of bounded graded norms in the same way, although 
the notation becomes rather cumbersome!

When $\cF$ is associated to a test configuration $(\cX,\cL)$ for $(X,L)$,
$\|\cF\|_p$ coincides with the $L^p$-norm of the test configuration as
defined in~\cite{Don05,WN12,His16}.
For a general bounded graded filtration, $\|\cF\|_2$ 
equals the $L^2$-norm considered by Sz\'ekelyhidi~\cite{Sze15}.

We say that a bounded graded norm (or filtration) is \emph{almost trivial} if its $p$th 
central moment is zero, \ie its limit measure is a Dirac mass. This means that it is equivalent to
a graded norm in the orbit of the trivial graded norm under the $\R$-action. By what 
precedes, a bounded graded filtration is almost trivial iff its $L^2$-norm in the 
sense of~\cite{Sze15} is zero.
%
%
%
%
%
%
%
%
%
\section{The asymptotic Fubini--Study operator}\label{sec:asFS}
In this section we define and study the asymptotic Fubini--Study operator.
This associates a bounded psh metric on $L$ to any graded norm
on the section ring $R=R(X,L)$.
In particular, we prove Theorems~B and~C in the introduction.
As in~\S\ref{s:filtnorm}, $L$ is an ample line bundle 
(as opposed to $\Q$-line bundle) on $X$.
%
%
%
%
\subsection{Psh metrics regularizable from below}
A bounded psh metric $\f\in\PSH(L)$ is \emph{regularizable from below} 
if there exists an increasing net $(\f_j)_{j\in J}$ of metrics in $\cH(L)$ 
converging to $\f$ in $\PSH(L)$. Thus $\f=\lim_j\f_j$ pointwise on $\Xqm$, and 
$\f$ is the usc regularization of the pointwise limit of the $\f_j$.
It is equivalent to demand the same property with the $\f_j$ being continuous psh
metrics.
We write $\PSH^\uparrow\!(L)$ for the set of metrics regularizable from below.
\begin{rmk}
  Inspired by the main result in~\cite{Bed80}, we conjecture that a psh metric 
  is regularizable from below iff it is continuous outside a pluripolar set,
  see Definition~\ref{D401}.
\end{rmk}
Not every bounded psh metric is regularizable from below.
\begin{exam}\label{Ex402}
  Let $X=\P^1_k$ and $L=\cO(1)$, 
  where $k$ is an infinite field, and let $(x_n)_1^\infty$ be a Zariski dense set 
  in $X(k)$. There exists a metric $\p\in\PSH(L)$ such that $\max\p=0$ 
  and $\p(x_n)=-\infty$
  for all $j$. (We can take $\p$ with $\MA(\p)=\sum_jc_n\d_{x_n}$ where $c_n>0$
  and $\sum_jc_n=1$). Then $\f:=\exp(\p)\in\PSH(L)$ is bounded,
  $\sup_X\f=1$, and $\f(x_n)=0$ for all $j$. 
  Suppose $\f=\lim_j\f_j$ for an increasing net $(\f_j)_j$ of positive metrics.
  Then $\f_j(x_n)\le\f(x_n)=0$ for all $j$ and $n$, so 
  since $\f_j$ is continuous and $x_n$ converges to the generic point of $X$,
  we must have $\f_j\le0$ for all $j$, and hence $\f\le0$,
  a contradiction.
\end{exam}
%
%
%
%
\subsection{Psh envelopes of bounded metrics}\label{S402}
Next we introduce
\begin{defi}\label{D402}
  For any bounded metric $\f$ on $L$ we set 
  \begin{equation}\label{e412}
    Q(\f)
    =\usc\sup\{\p\in\cH(L)\mid \p\le\f\},
  \end{equation}
\end{defi}
In~\eqref{e412} it would be equivalent to replace $\cH(L)$ by
$\cH(L)_\R$ or $\PSH(L)\cap C^0(X)$, since any continuous psh metric can
be uniformly approximated by positive metrics.
Also note that $Q(\f)=Q(\lsc\f)$. Indeed, if $\p\in\cH(L)$, then $\p\le\f$
iff $\p\le\lsc\f$.
\begin{lem}
  For any bounded metric $\f$, we have 
  $Q(\f)\in\PSH^\uparrow(L)$.
  Further, if $\f$ is usc, then $Q(\f)\le\f$, with equality iff $\f\in\PSH^\uparrow(L)$. 
\end{lem}
\begin{proof}
  The set $J$ of metrics $\p\in\cH(L)$ such that $\p\le\f$ is directed in the obvious way:
  $\p\le\p'$ iff $\p(x)\le\p'(x)$ for all $x\in X$.
  We can then view $J$ as an increasing net in $\cH(L)$, indexed by itself. This
  net converges to $Q(\f)$ in $\PSH(L)$, so $Q(\f)\in\PSH^\uparrow(L)$.
  The remaining statements are clear.
\end{proof}

\smallskip
In~\cite{trivval} we considered a different envelope: for any bounded 
metric $\f$ on $L$, set 
\begin{equation}\label{e414}
  P(\f)=\sup\{\p\in\PSH(L)\mid \p\le\f\}.
\end{equation}
In view of Example~\ref{Ex402}, we can have $P(\f)\ne Q(\f)$, even
when $\f\in\PSH(L)$. 
\begin{prop}\label{P404}
  Let $\f$ be a bounded metric on $L$.
  \begin{itemize}
  \item[(i)]
    If $\f$ is continuous, so if $P(\f)$, and $Q(\f)=P(\f)$;
  \item[(ii)]
    In general, $Q(\f)=\usc P(\lsc\f)$.
  \end{itemize}
\end{prop}
\begin{proof}
  First assume $\f$ is continuous. 
  By Corollary~5.28 in~\cite{trivval}, $P(\f)$ is a continuous psh metric.
  Since the supremum in~\eqref{e414} is taken over a larger set of metrics
  than in~\eqref{e412}, we have $Q(\f)\le\usc(P(\f))=P(\f)$.
  On the other hand, since $P(\f)$ can be uniformly approximated by positive 
  metrics, it follows from~\eqref{e412} that $Q(\f)\ge P(\f)$.
  This proves~(i).

  To prove~(ii), we may assume $\f$ is lsc, since $Q(\f)=Q(\lsc\f)$. 
  We must then prove that $Q(\f)=\usc P(\f)$. The inequality $Q(\f)\le\usc P(\f)$
  is clear from the definitions. 
  Being lsc, $\f$ is the limit of an increasing net $\f_j$ of continuous metrics on $L$.
  Pick $\e>0$ and suppose $\p\in\PSH(L)$ satisfies $\p\le\f$.
  Then $\p<\f_j+\e$ for $j\gg0$. Thus~(i) gives
  \begin{equation*}
    \p
    =P(\p)
    \le P(\f_j+\e)
    = P(\f_j)+\e
    =Q(\f_j)+\e
    \le Q(\f)+\e,
  \end{equation*}
  Taking the supremum over $\p$ yields $P(\f)\le Q(\f)+\e$, and hence
  $\usc P(\f)\le Q(\f)+\e$. 
\end{proof}
%
%
%
%
\subsection{The asymptotic Fubini--Study operator}
First suppose $L$ is globally generated, and consider a norm $\|\cdot\|$ on $H^0(X,L)$. 
Define a metric $\FS(\|\cdot\|)$ on $L$ by setting
\begin{equation*}
  \FS(\|\cdot\|):=\max\{\log(|s|_0/\|s\|)\mid s\in H^0(X,L)\setminus\{0\}\}.
\end{equation*}
This can more concretely be described as follows. Pick a basis $\{s_j\}_{j=1}^N$ for 
$H^0(X,L)$ that is orthogonal for $\|\cdot\|$ and write $\la_j:=-\log\|s_j\|$.
Then 
\begin{equation}\label{e413}
  \FS(\|\cdot\|)=\max_j(\log|s_j|_0+\la_j)
\end{equation}
Clearly $\FS(\|\cdot\|)$ belongs to $\cH(L)_\R$, and is hence 
a continuous psh metric.

Now assume $L$ is an ample bundle. Then $mL$ is globally generated for $m\gg0$.
\begin{lem}
  The set $\cH(L)_\R$ coincides with the set of metrics 
  of the form $m^{-1}\FS(\|\cdot\|)$,
  where $m\gg0$ and $\|\cdot\|$ is a norm on $R_m$.
\end{lem}
\begin{proof}
  It is clear from~\eqref{e413} that any metric of the form $m^{-1}\FS(\|\cdot\|)$
  belongs to $\cH(L)_\R$. Conversely, any metric $\f\in\cH(L)_\R$ is of the form
  $\f=m^{-1}\max_{1\le j\le N'}(\log|s_j|_0+\la_j)$, where $m\ge1$,
  $s_1,\dots,s_{N'}$ are global sections of $mL$ without common zero,
  and $\la_j\in\R$, $1\le j\le N'$. 
  We may assume $\la_1\ge\la_2\ge\dots\ge\la_{N'}$.
  We may also assume that $s_j$ does not lie in the linear span of 
  $(s_i)_{i<j}$ for $1\le j\le N'$, or else we could remove the entry $\log|s_j|+\la_j$
  from the max defining $\f$. 
  Now $(s_j)_{j\le N'}$ are linearly independent, and 
  can be completed to a basis $(s_j)_{j\le N_m}$ of $R_m$. 
  If we pick $\la_j\ll0$ for $j>N'$, then $\f=m^{-1}\max_{j\le N_m}(\log|s_j|+\la_j)$.
  Define a norm $\|\cdot\|$ on $R_m$ by $\|\sum a_js_j\|=\max_j|a_j|\exp(-\la_j)$.
  Then $\f=m^{-1}\FS(\|\cdot\|)$.
\end{proof}

Now consider a bounded graded norm $\|\cdot\|_\bullet$ on $R$. 
Then $\f_m:=m^{-1}\FS(\|\cdot\|_m)\in\cH(L)_\R$ for $m\gg0$.
The fact that $\|\cdot\|_\bullet$ is bounded means that there exists $C>0$ such that 
$|\f_m|\le C$ for all $m\gg0$.
The submultiplicative property $\|s\otimes s'\|_{m+m'}\le\|s\|_m\cdot\|s'\|_{m'}$ implies
$(m+m')\f_{m+m'}\ge m\f_{m}+m'\f_{m'}$ pointwise on $X$. By Fekete's lemma, the limit
$\lim_{m\to\infty}\f_m$ therefore exists pointwise on $X$, is finite, 
and equals $\sup_m\f_m$. Now set
\begin{equation*}
  \FS(\|\cdot\|_\bullet)
  :=\usc(\lim_{m\to\infty}\f_m)
  =\usc(\sup_m\f_m),
\end{equation*}
where $\usc$ denotes upper semicontinuous regularization.
Thus $\FS(\|\cdot\|_\bullet)$ is a bounded psh metric on $L$, 
called the 
\emph{asymptotic Fubini--Study metric} associated to $\|\cdot\|_\bullet$.
\begin{rmk}\label{R401}
  It will be useful to consider $\Z_{>0}$ as a set directed by $r_1\le r_2$ 
  iff $r_1\mid r_2$; then $(\f_r)_r$ is an increasing net in $\PSH(L)$ that converges to $\f$.
\end{rmk}
The next result shows how the asymptotic Fubini--Study operator 
interacts with natural operations on graded norms and metrics.
The proof is left to the reader. See Proposition~\ref{P421} for a 
deeper result.
\begin{prop}\label{P402}
  Let $\|\cdot\|_\bullet$ be a bounded graded norm on $R$,
  $c\in\R$, and $r\in\Z_{>0}$. Then:
  \begin{itemize}
  \item[(i)]
    $\FS(\exp(c\,\bullet)\|\cdot\|_\bullet)=\FS(\|\cdot\|_\bullet)-c$;
  \item[(ii)]
    if $\|\cdot\|^{(r)}_\bullet$ is the restriction of $\|\cdot\|_\bullet$ to
    $R(X,rL)$, then $\FS(\|\cdot\|_\bullet)=r^{-1}\FS(\|\cdot\|^{(r)}_\bullet)$
 \end{itemize}
\end{prop}
%
%
%
%
\subsection{The range of the asymptotic Fubini--Study operator}\label{sec:FSrange}
We now turn to the proof of Theorem~B in the introduction. 
It is clear that any asymptotic Fubini--Study
metric is regularizable from below. To prove the converse, we associate
to any (not necessarily psh) bounded metric $\f$ on $L$ its \emph{graded supremum norm} 
$\|\cdot\|_{\f,\bullet}$ defined by 
\begin{equation*}
  \|\cdot\|_{\f,m}:=\sup_X|\cdot|_0\exp(-m\f)
\end{equation*}
for $m\ge1$. This is a bounded graded norm on $R$, and it follows from 
the definition that 
\begin{equation}\label{e424}
  \|\cdot\|_{\f,\bullet}\vee\|\cdot\|_{\f,\bullet}=\|\cdot\|_{\f\wedge\f',\bullet}
\end{equation}
for bounded metrics $\f$ and $\f'$ on $L$. 
\begin{prop}\label{P403}
  For any bounded metric $\f$ on $L$ we have $Q(\f)=\FS(\|\cdot\|_{\f,\bullet})$.
\end{prop}
\begin{rmk}
  As a consequence, it would have sufficed to use sequences rather than nets
  in Definition~\ref{D402}.
  This can be viewed as an instance of Choquet's lemma.
\end{rmk}
\begin{proof}
  We will use the fact that~\eqref{e412} remains true with $\cH(L)$
  replaced by $\cH(L)_\R$. 
  Write $\f_m:=m^{-1}\FS(\|\cdot\|_{\f,m})$ for $m\ge1$ so that 
  $\FS(\|\cdot\|_{\f,\bullet})=\usc\sup_m\f_m$.
  On the one hand, $\f_m\in\cH(L)_\R$ and $\f_m\le\f$ for all 
  $m\gg0$, so $\FS(\|\cdot\|_{\f,\bullet})\le Q(\f)$. 
  On the other hand, suppose $\p\in\cH(L)_\R$ and $\p\le\f$.
  Write $\p=m^{-1}\max_{1\le j\le N}(\log|s_j|_0+\la_j)$, where $s_j\in R_m$, 
  $\la_j\in\R$ and the $s_j$ have no common zero.
  Since $\p\le\f$, we have $\|s_j\|_{m\f}=\sup|s_j|_0\exp(-m\f)\le\exp(-m\la_j)$,
  and hence $\log(|s_j|_0/\|s_j\|_{m\f})\ge\log|s_j|_0+\la_j$ for all $j$.
  Taking the max over $j$ gives 
  \begin{equation*}
    \p
    \le m^{-1}\max_j\log(|s_j|_0/\|s_j\|_{m\f})
    \le m^{-1}\FS(\|\cdot\|_{\f,m})
    =\f_m
    \le\FS(\|\cdot\|_{\f,\bullet}).
  \end{equation*}
  This gives $Q(\f)\le\FS(\|\cdot\|_{\f,\bullet})$ and concludes the proof.
\end{proof}
As a consequence, we obtain the following result, which implies Theorem~B.
\begin{cor}\label{C408}
  If $\f\in\PSH(L)$ is regularizable from below, then $\f=Q(\f)=\FS(\|\cdot\|_{\f,\bullet})$.
\end{cor}
The following result is more restrictive but also more precise.
\begin{lem}\label{L402}
  Suppose that $R$ is generated in degree 1, and that $\f=\f_\cL$ is 
  defined by a globally generated test configuration $(\cX,\cL)$ for $(X,L)$.
  Then $\f=Q(\f)=\FS(\|\cdot\|_{\f,1})$.
\end{lem}
\begin{proof}
  Since $\FS(\|\cdot\|_{\f,1})\le Q(\f)\le\f$, it suffices to prove
  $\FS(\|\cdot\|_{\f,1})=\f$.

  By assumption, $\cL$ defines a $\bbG_m$-equivariant 
  morphism of $\cX$ into a product test configuration 
  $\bbP\times\bbA^1_k$, where $\bbP=\bbP(H^0(X,L))$,
  and $\cL$ is the pullback of $\cO_{\bbP\times\bbA^1_k}(1)$. 
  Since $\f$ and $\FS(\|\cdot\|_{\f,1})$ are the pullbacks of the 
  corresponding metrics on $\cO_\P(1)$, it suffices to consider the 
  case when $X$ is a projective space, $L=\cO_\P(1)$, and the test
  configuration for $(X,L)$ is a product (but not necessarily trivial) test configuration.
  Pick a basis $s_0,\dots,s_N$ of $H^0(X,L)$ that diagonalizes the $\bbG_m$-action,
  with the weights being $\la_0\le\la_1\le\dots\le\la_N$.
  
  On the one hand, $\FS(\|\cdot\|_{\f,1})=\max_j\log|s_j|_0+\la_j$.
  On the other hand, for each $j$, $s_j$ extends as a $\bbG_m$-equivariant section
  $\bar{s}_j$ of $\cL$  over $\bbP\times\bbG_m$. Then $\unipar^{-\la_j}\bar{s_j}$, $1\le j\le N$
  are global sections of $\cL$ without common zero on $\bbP\times\bbA^1_k$.
  Thus the metric $\f_\cL$ defined by the test configuration $(\cX,\cL)$ is given by
  $\f_\cL=\max_j(|s_j|_0+\la_j)$, which completes the proof.
\end{proof}
%
%
%
%
\subsection{Monge--Amp\`ere energy and relative limit measures}
Next we show that the relative Monge--Amp\`ere energy of two 
asymptotic Fubini--Study metrics is equal to the relative volume of the
graded norms, as well as the barycenter of their relative limit measure.
\begin{thm}\label{T407}
  Let $\|\cdot\|_\bullet$ and $\|\cdot\|'_\bullet$ be bounded graded norms on $R=R(X,L)$,
  and set $\f:=\FS(\|\cdot\|_\bullet)$, $\f':=\FS(\|\cdot\|'_\bullet)$.
  Then we have 
  \begin{equation}\label{e430}
    E(\f,\f')
    =\int_\R\la\,d\nu(\la)
    =\vol(\|\cdot\|_\bullet,\|\cdot\|'_\bullet),
  \end{equation}
  where $\nu=\RLM(\|\cdot\|_\bullet,\|\cdot\|'_\bullet)$ is the relative limit measure
  of $\|\cdot\|_\bullet$ with respect to $\|\cdot\|'_\bullet$.
\end{thm}
\begin{cor}\label{C409}
  If $\f$ and $\f'$ are bounded metrics on $L$, then 
  \begin{equation}\label{e416}
    E(Q(\f),Q(\f'))
    =\vol(\|\cdot\|_{\f,\bullet},\|\cdot\|_{\f',\bullet}).
  \end{equation}
 \end{cor}
\begin{proof}
  This follows since $Q(\f)=\FS(\|\cdot\|_{\f,\bullet})$ and $Q(\f')=\FS(\|\cdot\|_{\f',\bullet})$.
\end{proof}
If $\f$ and $\f'$ are \emph{continuous} metrics on $L$, then by Proposition~\ref{P404},
the envelopes $Q(\f)$, $Q(\f')$
coincide with the envelopes considered in~\cite{BE18,BGJKM16}, namely $P(\f)$, $P(\f')$,
so Corollary~\ref{C409} recovers the main result of~\loccit~in the case of a trivially valued field. 
\begin{proof}[Proof of Theorem~\ref{T407}]
  The last equality in~\eqref{e430} follows from Corollary~\ref{C406}. 
  By the cocycle properties of $E$ and $\vol$,
  we may assume $\|\cdot\|'_\bullet$ is the trivial graded norm, and $\f'$
  the trivial metric on $L$.

  After replacing $L$ by a multiple, and using Propositions~\ref{P407} and~\ref{P402}, 
  we may assume that $L$ is a line bundle, and that 
  the canonical map $S^mR_r\to R_{mr}$ is surjective for all $r,m\ge1$.

  Let $\cF$ be the graded filtration associated to $\|\cdot\|_\bullet$. 
  It is easy to see that replacing $\cF$
  by the associated graded $\Z$-filtration $\cF_\Z$ (and modifying the graded norm
  accordingly) does not change $\f$, $\nu$, or $\vol(\|\cdot\|_\bullet)$.
  We may therefore assume that $\cF$ is a graded $\Z$-filtration.

  First assume $\cF$ is generated in degree 1. In this case, 
  $\cF$ is associated to an ample test configuration $(\cX,\cL)$ for $(X,L)$,
  and $\f=\f_\cL$ is the metric defined by $(\cX,\cL)$, see Lemma~\ref{L402}.
  The first equality in~\eqref{e430} then follows from~\cite[Proposition~5.9]{BHJ1}.

  Now consider a general bounded graded norm $\|\cdot\|_\bullet$ whose associated
  graded filtration $\cF$ is a graded $\Z$-filtration.
  Since the canonical map $S^mR_r\to R_{rm}$ is surjective for all $r,m\ge1$, 
  we can equip $R_{rm}$ with the quotient norm from the norm $\|\cdot\|_r$ on $R_r$. 
  This defines a bounded graded norm
  $\|\cdot\|_\bullet^{(r)}$ on the section ring $R(X,rL)$. 
  By definition, this graded norm is generated in degree 1. 
  Let $\cF^{(r)}$ be the associated graded filtration of $R(X,rL)$,
  $\nu_r:=r^{-1}_*\LM(\|\cdot\|^{(r)}_\bullet)$ its (scaled) limit measure,
  and $\f_r:=r^{-1}\FS(\|\cdot\|^{(r)}_\bullet)$ the corresponding metric on $L$. 
  By what precedes, we have $E(\f_r)=\int\la\,d\nu_r(\la)$ for each $r\ge1$.

  As in Remark~\ref{R401}, consider $\Z_{>0}$ as a set directed by $r_1\le r_2$ 
  iff $r_1\mid r_2$. Then $(\f_r)_r$ is an increasing net in $\PSH(L)$ that converges to $\f$.
  Since the Monge--Amp\`ere energy is continuous along increasing nets,
  we get $\lim_rE(\f_r)=E(\f)$.

  It remains to prove that $\lim_{r\to\infty}\int\la\,d\nu_r(\la)=\int\la\,d\nu(\la)$.
  Since the measures $\nu_r$ have
  support contained in a fixed compact interval, it suffices to prove that 
  $\lim_r\nu_r=\nu$ weakly.
  To do so, we recall the construction of the limit measure from~\cite{BC11}.
  For each $\la\in\R$, consider the graded subalgebra 
  $V^{(\la)}_\bullet$ of $R(X,L)$ defined by 
  $V^{(\la)}_m=\{s\in R_m\mid \|s\|_m\le\exp(-m\la)\}$.
  Using Okounkov bodies~\cite{LM09}, one shows that 
  the limit $\vol(V^{(\la)}_\bullet):=\lim_{m\to\infty}n!\vol(V^{(\la)}_m)/m^n$ exists,
  and is a decreasing function of $\la$. The limit measure then
  satisfies 
  \begin{equation*}
    \nu=-\vol(L)^{-1}\frac{d}{d\la}\vol(V^{(\la)}_\bullet)
  \end{equation*}
   in the sense of distributions.
  Similarly, for $r\ge1$ and $\la\in\R$, we have a graded subalgebra
  $W^{(r,\la)}_\bullet$ of $R(X,rL)$, given by 
  $W^{(r,\la)}_m=\{s\in R_{rm}\mid \|s\|^{(r)}_m\le\exp(-rm\la)\}$.
  Again, the limit $\vol(W^{(r,\la)}_\bullet):=\lim_{m\to\infty}n!\vol(W^{(r,\la)}_m)/m^n$ exists,
  and is a decreasing function of $\la$, and
  \begin{equation*}
    \nu_r
    =-\vol(rL)^{-1}\frac{d}{d\la}\vol(W^{(r,\la)}_\bullet)
    =-\vol(L)^{-1}\frac{d}{d\la}r^{-n}\vol(W^{(r,\la)}_\bullet).
  \end{equation*}
  Set $g(\la):=\vol(V_\bullet^{(\la)})$ and $g_r(\la):=r^{-n}\vol(W^{(r,\la)}_\bullet)$. 
  By Lemma~\ref{L412} below, $g_r\to g$ pointwise on $\R$.
  Since $0\le g_r,g\le\vol(L)$, it follows from dominated convergence that 
  $g_r\to g$ in $L^1_{\mathrm{loc}}(\R)$. Hence $\nu_r\to\nu$ in the sense
  of distributions, which completes the proof.
\end{proof}
\begin{lem}\label{L412}
  Let $V_\bullet\subset R(X,L)$ be a graded subalgebra.
  Suppose that for $r\ge1$ we have a graded subalgebra
  $W^{(r)}_\bullet\subset V^{(r)}_\bullet$, where $V^{(r)}_m:=V_{rm}$ for $m\ge1$,
  satisfying $W^{(r)}_1=V^{(r)}_1$. 
  Then $\lim_{r\to\infty}r^{-n}\vol(W^{(r)}_\bullet)=\vol(V_\bullet)$.
\end{lem}
\begin{proof}
  We use Okounkov bodies, following~\cite{Bou14}.
  Pick a valuation $\mu\colon k(X)^\times\to\Z^n$ of rational rank $n$.
  Set $\Gamma_m:=\mu(V_m\setminus\{0\})$ and 
  $\Gamma^{(r)}_m:=\mu(W^{(r)}_m\setminus\{0\})$ for $m\ge1$. 
  Let $\D(V_\bullet)$ and $\D(W^{(r)}_\bullet)$ be the 
  closed convex hull inside $\R^n$ of $\bigcup_mm^{-1}\Gamma_m$
  and of $\bigcup_mm^{-1}\Gamma^{(r)}_m$, respectively.
  Then $\vol(V_\bullet)=n!\vol(\D(V_\bullet))$ and 
  $\vol(W^{(r)}_\bullet)=n!\vol(\D(W^{(r)}_\bullet))$,
  so it suffices to prove that 
  $\lim_{r\to\infty}\vol(r^{-1}\D(W^{(r)}_\bullet))=\vol(\D(V_\bullet))$.

  Since $W^{(r)}_m\subset V_{rm}$, we get 
  $\Gamma^{(r)}_m\subset\Gamma_{rm}$ for all $r,m$, and hence
  $r^{-1}\D(W^{(r)}_\bullet)\subset\D(V_\bullet)$ for all $r$.
  If $\vol(\D(V_\bullet))=0$, we are done, so we may assume $\D(V_\bullet)$
  has nonempty interior. Pick compact subsets $K$ and $L$ of $\R^n$
  with $K\subset\joinrel\subset L\subset\joinrel\subset\D(V_\bullet)$. 
  It suffices to prove that $r^{-1}\D(W^{(r)}_\bullet)\supset K$ for $r\gg1$.
  Now $r^{-1}\Z^n\cap L=r^{-1}\Gamma_r\cap L$ for $r\gg1$,
  see~\cite[Lemme~1.13]{Bou14}.
  If $\D_r$ is the convex hull of $r^{-1}\Gamma_r$, it follows that 
  $\D_r\supset K$ for $r\gg1$.
  But $W^{(r)}_1=V_r$, so $\Gamma^{(r)}_1=\Gamma_r$, 
  and hence $r^{-1}\D(W^{(r)}_\bullet)\supset\D_r\supset K$, 
  which completes the proof.
\end{proof}
%
%
%
%
\subsection{Injectivity of the asymptotic Fubini--Study operator}
The next result is a reformulation of Theorem~C in the introduction. 
\begin{thm}\label{T412}
  Let $\|\cdot\|_\bullet$ and $\|\cdot\|'_\bullet$ be bounded graded norms on $R(X,L)$
  and write $\f:=\FS(\|\cdot\|_\bullet)$, $\f':=\FS(\|\cdot\|'_\bullet)$.
  Then $\|\cdot\|_\bullet$ and $\|\cdot\|'_\bullet$ are equivalent iff $\f=\f'$.
\end{thm}
We will prove this using the Monge--Amp\`ere energy. For this we need 
some preparation.
\begin{lem}\label{L405}
  If $\f,\f'\in\cE^1(L)$ and $\f\ge\f'$, then $\f=\f'$ iff $E(\f,\f')=0$.
\end{lem}
\begin{proof}
  The direct implication is trivial. For the reverse direction, note that 
  \begin{equation*}
    E(\f,\f')
    =\frac{1}{(n+1)V}\sum_{j=0}^n\int(\f'-\f)\omega_\f^j\wedge\omega_\p^{n-j}
    \ge\frac1{n+1}I(\f,\f').
  \end{equation*}
  This implies $I(\f,\f')=0$, and hence $\f-\f'$ is constant, say $\f=\f'+c$, with $c\ge0$.
  But then $c=E(\f,\f')=0$, completing the proof.
\end{proof}
\begin{cor}\label{C410}
  The assertion of Theorem~\ref{T412} holds when 
  $\|\cdot\|_\bullet\le\|\cdot\|'_\bullet$.
\end{cor}
\begin{proof}
  Indeed, in this case, $\f\ge\f'$, and~\eqref{e426},~\eqref{e430} imply that
  \begin{equation*}
    d_1(\|\cdot\|_\bullet,\|\cdot\|'_\bullet)
    =\vol(\|\cdot\|_\bullet,\|\cdot\|'_\bullet)
    =E(\f,\f'),
  \end{equation*}
  which concludes the proof in view of Lemma~\ref{L405}.
\end{proof}
\begin{lem}\label{L413}
  If $\|\cdot\|_\bullet$ is a bounded graded norm, and $\f:=\FS(\|\cdot\|_\bullet)$
  its asymptotic Fubini--Study metric,
  then the supremum graded norm $\|\cdot\|_{\f,\bullet}$ is equivalent to $\|\cdot\|_\bullet$.
  It is the smallest bounded graded norm equivalent to $\|\cdot\|_\bullet$.
\end{lem}
\begin{proof}
  We first prove that $\|\cdot\|_{\f,\bullet}\le\|\cdot\|_\bullet$.
  Recall that $\f=\sup_m\f_m$, where
  $\f_m:=m^{-1}\sup_{s\in R_m\setminus\{0\}}\log\frac{|s|_0}{\|s\|_m}$.
  Thus, for every $s\in R_m\setminus\{0\}$, 
  \begin{equation*}
    \|s\|_{\f,m}
    =\sup_X|s|_0\exp(-m\f)
    \le\sup_X|s|_0\exp(-mm^{-1}\log\frac{|s|_0}{\|s\|_m})
    =\|s\|_m.
  \end{equation*}
  Since $\f\in\PSH^\uparrow(L)$, 
  Corollary~\ref{C408} shows that $\FS(\|\cdot\|_{\f,\bullet})=\f=\FS(\|\cdot\|_\bullet)$.
  Thus $\|\cdot\|_\bullet$ and $\|\cdot\|_{\f,\bullet}$ are equivalent by Corollary~\ref{C410}.
\end{proof}
\begin{prop}\label{P421}
  Let $\|\cdot\|_\bullet$ and $\|\cdot\|'_\bullet$  be bounded graded norms on $R$.
  Then 
  \begin{equation}\label{e427}
    \FS(\|\cdot\|_\bullet\vee\|\cdot\|'_\bullet)=Q(\f\wedge\f'),
  \end{equation}
  where $\f:=\FS(\|\cdot\|_\bullet)$ and $\f':=\FS(\|\cdot\|'_\bullet)$.
\end{prop}
\begin{proof}
  By Lemma~\ref{L413}, the graded norms $\|\cdot\|_\bullet$ and $\|\cdot\|_{\f,\bullet}$
  are equivalent, as are $\|\cdot\|'_\bullet$ and $\|\cdot\|_{\f',\bullet}$.
  Lemma~\ref{L411} then shows that 
  $\|\cdot\|_\bullet\vee\|\cdot\|'_\bullet$ is equivalent to 
  $\|\cdot\|_{\f,\bullet}\vee\|\cdot\|_{\f',\bullet}$.
  Thus we may assume 
  $\|\cdot\|_\bullet=\|\cdot\|_{\f,\bullet}$ and $\|\cdot\|'_\bullet=\|\cdot\|_{\f',\bullet}$.
  Now $\|\cdot\|_\bullet\vee\|\cdot\|'_\bullet=\|\cdot\|_{\f\wedge\f',\bullet}$, see~\eqref{e424},
  so the result follows from Proposition~\ref{P403}.
\end{proof}
\begin{cor}\label{C405}
  If $\|\cdot\|_\bullet$ and $\|\cdot\|'_\bullet$ are bounded graded norms on $R(X,L)$,
  then 
  \begin{equation}\label{e431}
    d_1(\|\cdot\|_\bullet,\|\cdot\|'_\bullet)
    =E(\f,Q(\f\wedge\f'))+E(\f',Q(\f\wedge\f')),
  \end{equation}
  where $\f:=\FS(\|\cdot\|_\bullet)$ and $\f':=\FS(\|\cdot\|'_\bullet)$.
\end{cor}
\begin{proof}
  This follows from Theorem~\ref{T407},~\eqref{e426} and~\eqref{e427}.
\end{proof}
\begin{proof}[Proof of Theorem~\ref{T412}]
  By definition,  the graded norms are equivalent iff 
  $d_1(\|\cdot\|_\bullet,\|\cdot\|'_\bullet)=0$. This
  happens iff both terms in the right hand side of~\eqref{e431} vanish,
  since the terms are nonnegative.
  By Lemma~\ref{L405}, this occurs iff $\f=Q(\f\wedge\f')=\f'$.
\end{proof}
We also mention a related result.
Let $\cF$ be a bounded graded filtration, with associated graded norm $\|\cdot\|_\bullet$
and metric $\f=\FS(\|\cdot\|_\bullet)\in\PSH^\uparrow(L)$.
Then
\begin{equation}\label{e435}
  c_n J(\f)
  \le
  \|\cF\|_1
  \le 2J(\f), 
\end{equation}
where $c_n=2n^n/(n+1)^{n+1}$ and $\|\cdot\|_1$ is the first central moment
defined in~\S\ref{S407}.
Indeed, by~\cite[Theorem~7.9]{BHJ1}, this holds when $\cF$ is 
induced by an ample test configuration, and the general case is 
reduced to this case by approximation, as in the proof of Theorem~\ref{T407}.
%
%
%
%
\subsection{The Darvas distance}\label{S406}
If $\f,\f'\in\PSH(L)$ are regularizable from below, then we set 
\begin{equation*}
  d_1(\f,\f'):=E(\f,Q(\f\wedge\f'))+E(\f',Q(\f\wedge\f')).
\end{equation*}
It follows that $d_1$ defines a distance on the space $\PSH^\uparrow(L)$.
Further, the asymptotic Fubini--Study operator gives an isometric bijection
between the space of equivalence classes of bounded graded norms on $R$, 
and the space $\PSH^\uparrow(L)$.

The metric on $\PSH^\uparrow(L)$ is analogous to the metric on $\cE^1(L)$
considered by Darvas in the Archimedean case. Indeed, in our setting,
if $\f$ and $\f'$ are continuous, then 
$d_1(\f,\f'):=E(\f,P(\f\wedge\f'))+E(\f',P(\f\wedge\f'))$,
in accordance with~\cite[Corollary~4.14]{Dar15}. A positive answer to the following
conjecture would make the analogy even stronger.
\begin{conj}
  We have $P(\f\wedge\f')=Q(\f\wedge\f')$ for $\f,\f'\in\PSH^\uparrow(L)$.
\end{conj}
%
%
%
%
%
%
%
%
 \subsection{K-stability and filtrations}\label{S408}
Sz\'ekelyhidi proved in~\cite{Sze15} that if the (reduced) automorphism group of $(X,L)$ 
is finite and 
$X$ admits a cscK metric in $c_1(L)$, then $(X,L)$ satisfies a condition---involving
filtrations---that is stronger than the usual notion of $K$-stability. 

For simplicity assume $R(X,L)$ is generated in degree 1.  
Consider a bounded graded $\Z$-filtration $\cF$ of $R(X,L)$.
For $r\ge 1$,
the filtration of $R_r$ induces a filtration $\cF^{(r)}$ of $R(X,rL)$  that is 
generated in degree one. This filtration defines an ample
test configuration $\cL_r$ for $L$.
Let $\DF(\cL_r)$ be its Donaldson--Futaki invariant, and set 
$\Fut(\cF):=\liminf_r\DF(\cL_r)$. 

Now suppose $(X,L)$ is uniformly K-stable
in the strong sense that there exists $\e>0$ such that $M(\f)\ge\e J(\f)$
for all $\f\in\cE^1(L)$. 
Pick $\f=\FS(\|\cdot\|_\bullet)$, where $\|\cdot\|_\bullet$ is the graded norm 
associated to $\cF$. By~\eqref{e435}, we have
$M(\f)\ge\frac{\e}{2}\|\cF\|_1$.
For $r\ge1$, $\f_r:=r^{-1}\FS(\|\cdot\|_r)$ is the positive metric
associated to $\cL_r$.
As in Remark~\ref{R401}, partially order $\Z_{>0}$ by $r\le r'$ iff $r\mid r'$.
Then the net $(\f_r)_r$ increases to $\f$.
Thus $\lim_rE(\f_r)=E(\f)$, and similarly $E_{K_X}(\f_r)\to E_{K_X}(\f)$.
Further $\MA(\f_r)\to\MA(\f)$, which implies $\liminf_r H(\f_r)\ge H(\f)$,
since $H(\f)=\Ent(\MA(\f))$, and $\Ent$ is lsc. In view of~\eqref{e432},
this gives $\liminf_r M(\f_r)\ge M(\f)\ge\frac{\e}{2}\|\cF\|_1$. 
Now $\DF(\cL_r)\ge M(\f_r)$ by~\cite[(7.7)]{BHJ1}, so we get 
\begin{equation*}
  \Fut(\cF)\ge\tfrac{\e}{2}\|\cF\|_1.
\end{equation*}
Since $\|\cF\|_2=0$ iff $\|\cF\|_1=0$, we conclude that 
$\Fut(\cF)>0$ whenever $\|\cF\|_2>0$; this is the condition
considered in~\cite{Sze15}.

If the (reduced) automorphism group of $(X,L)$ 
is finite and $X$ admits a cscK metric in $c_1(L)$, then $(X,L)$ is 
uniformly K-stable in the usual sense, see~\cite[Theorem~1.7]{BDL16}.
%
%
%
%
%
%
\section{A valuative criterion of K-stability}\label{S405}
In~\cite{FujitaValcrit}, Fujita gave a valuative criterion of K-semistability
and uniform K-stability of Fano varieties. The case of K-semistability was treated 
independently by C.~Li~\cite{LiEquivariant}.
Here we use psh metrics to prove a version of the valuative
criterion in the general adjoint setting of an ample $\Q$-line bundle $L$
on a smooth projective variety $X$. 
In particular, we prove Theorems~D,~E and~E' of the introduction.
%
%
%
%
\subsection{Alexander-Taylor capacity}
Inspired by~\cite{AT84} we define  the \emph{Alexander-Taylor capacity}\footnote{It is the ``multiplicative'' version $e^{-T(x)}$ of $T(x)$ that behaves like the capacity $T$ in~\cite{AT84}.}
of a point $x\in X$ (relative to $L$) as 
\begin{equation*}
  T(x)=\sup\{\sup_X\f-\f(x)\mid \f\in\PSH(L)\}.
\end{equation*}
Here it is equivalent to take the supremum over metrics $\f\in\cH(L)$ 
and/or normalized by $\f(x)\le 0$. 
The following definition will be used in what follows.
\begin{defi}\label{D401}
  A subset $E\subset X$ is \emph{pluripolar} if there exists $\f\in\PSH(L)$ 
  with $\f|_E\equiv-\infty$.
\end{defi}
This notion does not depend on the choice of ample $\Q$-line bundle $L$.
Indeed, if $L'$ is another ample $\Q$-line bundle, then there exists $\e>0$ rational
such that $L'':=L'-\e L$ is ample. If $\f''\in\cH(L'')$, then 
$\f':=\e\f+\f''\in\PSH(L'')$ and $\f'\equiv-\infty$ on $E$.

We say that a point $x\in X$ is pluripolar if $\{x\}$ is pluripolar. 
Any point $x\in X\setminus\Xval$ is pluripolar. 
Indeed, there exists a subvariety $Y\subsetneq X$ containing $x$. Since $L$ is ample,
there exists $m\ge1$ and a nonzero section $s\in H^0(X,mL)$ vanishing along $Y$.
Then $\f:=m^{-1}\log|s|_0\in\PSH(L)$ and $\f(x)=-\infty$.
There also exist pluripolar points $x\in\Xval$.
A simple example from~\cite{ELS} is explained in~\cite[Remark~2.19]{BKMS}.
\begin{prop}\label{P406}
  Let $x\in X$ be any point. Then:
  \begin{itemize}
  \item[(i)]
    $T(x)\ge0$, with equality iff $x$ is the generic point of $X$;
  \item[(ii)]
    $T(x)=\infty$ iff $x$ is pluripolar.
  \end{itemize}
\end{prop}
\begin{proof}[Proof of Proposition~\ref{P406}]
  Clearly, $T(x)\ge0$.
  If $x$ is the generic point, then $\f(x)=\sup_X\f$ for every $\f\in\PSH(L)$,
  so $T(x)=0$. Now suppose $x$ is not the generic point of $X$ and set 
  $\xi=\redu(x)$. There exists $m\ge1$ and a nonzero section $s\in H^0(X,mL)$
  vanishing at $\xi$. Then $\f:=m^{-1}\log|s|_0\in\PSH(L)$ satisfies
  $\sup_X\f=0$ and $\f(x)<0$, so $T(x)>0$. This proves~(i).

  As for~(ii), it is clear that $T(x)=\infty$ when $x$ is pluripolar.
  Conversely, if $T(x)=\infty$, there exists a sequence $(\f_j)_1^\infty$
  of metrics in $\PSH(L)$ such that $\sup_X\f_j=0$ and $\f_j(x)\le-2^j$.
  Then $\f:=\sum_j2^{-j}\f_j\in\PSH(L)$ satisfies $\f(x)=-\infty$,
  so $x$ is pluripolar.
\end{proof}
We shall later use the following description of the Alexander-Taylor capacity.
\begin{lem}\label{L404}
  Given a subset $P\subset\PSH(L)$, set $T_P(x):=\sup_{\f\in P}\{\sup_X\f-\f(x)\}$.
  Then we have $T_{\cH(L)}=T_{P_1}(x)=T_{P_2}(x)=T(x)$, where 
  \begin{equation*}
    P_1=\{\f=m^{-1}\log|s|_0\mid m\ge 1, s\in H^0(X,mL)\setminus\{0\}\},
  \end{equation*}
  and $P_2=\{f=\p+r\mid \p\in P_1, r\in\Q\}$.
\end{lem}
\begin{proof}
  That $T_{\cH(L)}(x)=T_{\PSH(L)}(x)$ is clear since 
  every function in $\PSH(L)$ is a decreasing limit of functions in $\cH(L)$. 
  That $T_{P_1}(x)=T_{P_2}(x)$ is also clear since $\f\mapsto\sup_X\f-\f(x)$ is
  translation invariant. 
  Finally, since every function in $\cH(L)$ is the maximum of finitely many functions in $P_2$,
  it follows that $T_{\cH(L)}=T_{P_2}(x)$. This completes the proof.
\end{proof}
%
%
%
%
\subsection{Monge--Amp\`ere energy}
The \emph{Monge--Amp\`ere energy} of a point $x\in X$ is 
\begin{equation*}
  S(x):=E^*(\d_x)\in\R_+\cup\{+\infty\}.
\end{equation*}
We can write this as
\begin{equation*}
  S(x)=\sup\{E(\f)-\f(x)\mid \f\in\cE^1(L)\},
\end{equation*}
where we may equivalently take the supremum over $\f\in\cH(L)$.
For such $\f$ we have 
\begin{equation}\label{e429}
  \frac{n}{n+1}\min\f+\frac{1}{n+1}\sup_X\f
  \le E(\f)
  \le\sup_X\f,
\end{equation}
so that $(n+1)^{-1}T(x)\le S(x)\le T(x)$.
(See~\eqref{e409} for a better estimate.)
In particular,
\begin{lem}
  For any $x\in X$, we have $S(x)=\infty$ iff $T(x)=\infty$ iff $x$ is pluripolar.
\end{lem}
The following result will be used later on.
\begin{prop}\label{P411}
  We have $E^*(\mu)\le\int S(x)\,d\mu(x)$ for any 
  Radon probability measure $\mu$ on~$X$.
\end{prop}
\begin{proof}
  For any $x\in X$ and any $\f\in\cH(L)$, we have 
  $E(\f)\le S(x)+\f(x)$. Integrating with respect to $x$, we get
  $E(\f)\le\int S(x)\,d\mu(x)+\int\f\,d\mu$, and taking the
  supremum over $\f\in\cH(L)$ gives the desired inequality.
\end{proof}
%
%
%
%
\subsection{Continuity properties}
We now study $S$ and $T$ as functions on $X$.
\begin{prop}\label{P410}
  Any point in $\Xqm$ is nonpluripolar. 
  Further, $S$ and $T$ are continuous on the dual complex $\D_\cX$ of
  any snc test configuration $\cX$ for $X$.
\end{prop}
\begin{proof}
  For any $\cX\in\SNC(X)$, it follows from~\cite[Theorem~5.29]{trivval}
  that the restriction to $\D_\cX$ of the family $\{\f-\sup_X\f\mid \f\in\cH(L)\}$
  is equicontinuous. In view of~\eqref{e429}, the same holds for 
  the family $\{\f-E(\f)\mid \f\in\cH(L)\}$. The result follows.
\end{proof}
\begin{rmk}
  There exist nonpluripolar points in $\Xval\setminus\Xqm$. See~Example~2.14 
  in~\cite{BKMS},   and also Example~\ref{Ex401} below.
\end{rmk}
\begin{lem}\label{L409}
  For any $x\in X$, the nets $(S(\retr_\cX(x)))_{\cX\in\SNC(X)}$ and 
  $(T(\retr_\cX(x)))_{\cX\in\SNC(X)}$ are increasing, with limits
  $S(x)$ and $T(x)$, respectively.
\end{lem}
\begin{proof}
  For every $\f\in\PSH(L)$, the net $(\f(\retr_\cX(x)))_\cX$ 
  is decreasing, with limit $\f(x)$. This implies both results.
\end{proof}
\begin{cor}
  The functions $S$ and $T$ are lsc on $X$.
\end{cor}
We also note
\begin{prop}
  The functions $S$ and $T$ are homogeneous on $X$ in the sense that 
  $S(x^t)=tS(x)$ and $T(x^t)=tT(x)$ for $x\in X$ and $t\in\R_+^\times$.
\end{prop}
\begin{proof}
  By~\cite[Proposition~7.12]{trivval} we have
  $S(x^t)=E^*(t\cdot\d_x)=tE^*(\d_x)=tS(x)$.
  Second,
  if $\f\in\PSH(L)$, the function $\f_t$ on $X$ defined by 
  $\f_t(y)=t\f(y^{1/t})$ also lies in $\PSH(L)$, see~\cite[Proposition~5.13]{trivval}.
  Further, $\f_t\le 0$ when $\f\le0$. 
  This easily implies $T(x^t)\ge tT(x)$. Replacing $(x,t)$
  by $(x^t,t^{-1})$ shows that equality must hold.
\end{proof}
%
%
%
%
\subsection{Associated graded norm}
If $L$ is an ample line bundle (as opposed to $\Q$-line bundle),
any point $x\in\Xval$ defines
a graded norm $\|\cdot\|_{x,\bullet}$ on $R=R(X,L)$ given by 
\begin{equation*}
  \|s\|_m=|s|_0(x)
\end{equation*} 
for $s\in R_m$.
Its associated graded filtration
was studied in detail in~\cite{BKMS}. 
\begin{prop}
  The graded norm $\|\cdot\|_{x,\bullet}$ is bounded iff $x$ is nonpluripolar.
  Further, the Alexander-Taylor capacity $T(x)$ coincides with the invariant 
  $T$ in~\cite{BlumJonsson}.
\end{prop}
The latter invariant was already introduced (but denoted differently) in~\cite{BKMS}.
It is called the \emph{maximum order of vanishing} in~\cite{BlumJonsson}.
We will see below that the Monge--Amp\`ere energy $S(x)$ coincides with 
the \emph{expected order of vanishing} as studied in~\cite{BKMS,BlumJonsson}.
\begin{proof}
  For $m\ge 1$, consider the successive minima
  $0=\la_{m,1}\le\la_{m,2}\le\dots\le\la_{m,N_m}$
  of the norm $\|\cdot\|_{x,m}$ on $R_m$. 
  The maximum order of vanishing in~\cite{BKMS,BlumJonsson} is
  then given by $\lim_{m\to\infty}m^{-1}\la_{m,N_m}=\sup_mm^{-1}\la_{m,N_m}$,
  and hence equals $T(x)$ in view of Lemma~\ref{L404}.
\end{proof}
%
%
%
%
\subsection{Metrics associated to points}
Given any point $x\in X$, define 
\begin{equation*}
  \f_x:=\sup\{\f\in\PSH(L)\mid \f(x)\le 0\}.
\end{equation*}
This is a function on $X$ with values in $[0,+\infty]$. We shall see that the behavior
of $\f_x$ is vastly different, depending upon whether $x$ is pluripolar or not.
\begin{prop}
  If $x$ is pluripolar, then the usc envelope of $\f_x$ is equal to $+\infty$ on $X$.
\end{prop}
\begin{proof}
  Since $\Xqm$ is dense in $X$, 
  it suffices to prove that $\f_x(y)=+\infty$ for every point $y\in\Xqm$.
  But by Izumi's inequality~\cite[Theorem~2.21]{trivval}, 
  there exists a constant $C=C(y)>0$
  such that $\f(y)\ge\sup_X\f-C$ for all $\f\in\PSH(L)$.
  This implies that $\f_x(y)=+\infty$.
\end{proof}
The following result is a more precise version of Theorem~D.
\begin{thm}\label{T410}
  Let $L$ be an ample $\Q$-line bundle.
  For any nonpluripolar point $x\in X$, $\f_x$ is a continuous psh metric on $L$,
  and satisfies $\MA(\f_x)=\d_x$, $\f_x(x)=0$. Further,
  \begin{equation}\label{e406}
    E(\f_x)=S(x),
    \quad
    I(\f_x)=T(x),
    \quad\text{and}\quad
    J(\f_x)=T(x)-S(x).
  \end{equation}
  When $L$ is a line bundle,  $\f_x$ is the asymptotic Fubini--Study
  metric of the graded norm $\|\cdot\|_{x,\bullet}$ defined by $x$,
  and we have $S(x)=\vol(\|\cdot\|_{x,\bullet})$.
\end{thm}
Before starting the proof, we make a few remarks.
First, the continuity of $\f_x$ can be interpreted as the singleton $\{x\}$ 
being a \emph{regular} compact set in the sense of pluripotential theory.
Second, the last equality shows that $S(x)$ agrees with the 
expected order of vanishing in~\cite{BKMS,BlumJonsson},
see also~\cite{MR15}.
Third, combining~\eqref{e406} and~\eqref{e410} gives
\begin{equation}\label{e409}
  \frac1{n+1}T(x)\le S(x)\le\frac{n}{n+1}T(x).
\end{equation}
for any nonpluripolar point $x\in X$. In fact, this also holds
when $x$ is pluripolar, since then $S(x)=T(x)=\infty$.
The inequality~\eqref{e409} is reminiscent of the Alexander-Taylor 
inequality~\cite[Theorem~2.1]{AT84}.
It was proved in~\cite{FujitaPlt} for a divisorial point.
The proof here is quite different; however, both proofs ultimately reduce to the 
Hodge Index Theorem.
\begin{proof}[Proof of Theorem~\ref{T410}]
  During the proof we denote the generic point of $X$ by $x_g$.

  First assume $x$ is quasimonomial. To prove that $\f_x$ is continuous and that 
  $\MA(\f_x)=\d_x$, we can use ground field extension to reduce to~\cite[\S8.4]{nama},
  which treats the case of a discretely valued ground field, but let us give a direct proof,
  based on the argument in~\loccit

  Given any open neighborhood $U$ of $x$ in $X$, we may find an snc test configuration 
  $\cX$ for $X$, with dual complex $\D=\D_\cX$, such that $\retr_\cX^{-1}(\{x\})\subset U$.
  This follows from the homeomorphism $X\simto\varprojlim_\cX\D_\cX$.
  Let $\D'$ be a (possibly irrational) 
  simplicial subdivision of $\D$ such that $x$ is a vertex of $\D'$. (Thus $\D'$ and 
  and $\D$ have the same underlying space.)
  Fix a constant $C\gg0$ and let $f_x$ be
  the unique function on $\D'$ that is affine on each face, $f_x(x)=0$, and 
  $f_x(y)=C$ for every vertex $y$ of $\D'$ with $y\ne x$. Extend $f_x$ to
  all of $X$ by demanding $f_x=f_x\circ\retr_\cX$. Then $f_x$ is continuous.

  We claim that $\f_x=P(f_x)$, where $P(f_x)=\sup\{\f\in\PSH(L)\mid \f\le f_x\}$
  is the psh envelope of $f_x$, see~\S\ref{S402}. 
  To see this, note that $P(f_x)\le\f_x$, so  it suffices to 
  prove that if $\f\in\PSH(L)$ and $\f(x)\le 0$, then $\f\le f_x$. 
  Now, if $C\gg0$, it follows from the Izumi estimate in~\cite[Theorem~2.21]{trivval} 
  that $\f(y)<C=f_x(y)$ for all vertices $y\ne x$ of $\D'$.
  Further, $\f$ is convex on each face of $\D$
  (see~\cite[Theorem~5.29]{trivval}), and hence also on each face of $\D'$,
  so we must have $\f\le f_x$ on $\D$, with equality only at $x$. 
  Since $\f\le\f\circ\retr_\cX$ by
  \textit{loc.\,cit.}, this gives $\f\le \f_x$ on $X$, with strict inequality outside 
  $\retr_\cX^{-1}\{x\}\subset U$.

  Now $f_x$ is continuous, so $\f_x=P(f_x)$ is also continuous by~\cite[Corollary~5.28]{trivval}.
  By the orthogonality property (see~\cite[Theorem~6.30]{trivval}), the support of $\MA(\f_x)$
  is contained in the locus where $\f_x=f_x$. But this locus is contained in $U$, which was
  an arbitrary neighborhood of $x$, so we must have $\MA(\f_x)=\d_x$, and also $\f_x(x)=0$.

  We further have $\f_x(x_g)=\sup_X\f_x=T(x)$. Indeed, 
  $T(x)\ge\f_x(x_g)-\f_x(x)=\f_x(x_g)$, and if $\f\in\cH(L)$, then $\f-\f(x)\le\f_x$,
  so $\f(x_g)-\f(x)\le\f_x(x_g)$, and hence $T(x)\le\f_x(x_g)$, so that 
  $T(x)=\f_x(x_g)$. 
  Using $\MA(\f_x)=\d_x$, this implies $I(\f_x)=\f_x(x_g)-\f_x(x)=T(x)$.
  We also have $(I-J)(\f_x)=E^*(\d_x)=S(x)$. As a consequence, 
  $J(\f_x)=T(x)-S(x)$, and hence $E(\f_x)=J(\f_x)-\f_x(x_g)=S(x)$. 

  \smallskip
  Now consider a general nonpluripolar point $x\in X$.  
  Thus $S(x),T(x)<\infty$. To simplify notation, denote by $J:=\SNC(X)$ 
  the directed set of snc test configurations for $X$. For $j\in J$,
  we have a retraction $r_j\colon X\to X$ onto the corresponding
  dual complex. Set $x_j:=r_j(x)$. Then $\lim_j x_j=x$, and 
  by Lemma~\ref{L409}, $S(x_j)$ and $T(x_j)$
  increase to $S(x)$ and $T(x)$, respectively.

  The net $(\f_{x_j})_j$ is increasing and bounded from above by $T(x)$;
  hence it converges in $\PSH(L)$ to the solution
  $\f'_x$ of $\MA(\f'_x)=\d_x$, normalized by $\sup_X\f'_x=T(x)$.
  Thus $E(\f_{x_j})\to E(\f'_x)$.  
  We claim that $\f'_x=\f_x$. 
  To see this, note that since $\MA(\f'_x)=\d_x$, we have
  \begin{equation*}
    S(x)
    =E^*(\d_x)
    =E(\f'_x)-\int\f'_x\d_x
    =E(\f'_x)-\f'_x(x)
  \end{equation*}
  and similarly $S(x_j)=E(\f_{x_j})-\f_{x_j}(x_j)=E(\f_{x_j})$.
  Since $\lim_jS(x_j)=S(x)$ and $\lim_jE(\f_j)=E(\f'_x)$, we 
  conclude that $\f'_x(x)=0$. 
  Now consider any metric $\f\in\PSH(L)$ with $\f(x)\le0$.
  Set $\p:=\max\{\f,\f'_x\}$. Then $\p\ge\f'_x$ so $E(\p)\ge E(\f'_x)$.
  Further, $\p(x)=0$, so 
  \begin{equation*}
    E(\p)-\int\p\d_x\ge E(\f'_x)=E^*(\d_x).
  \end{equation*}
  Thus $\p$ maximizes the  functional $\f\mapsto E(\f)-\int\f\d_x$, 
  whose supremum equals $E^*(\d_x)$ and is also maximized by $\f'_x$.
  Since the maximizer is unique up to an additive constant, and since
  $\p(x)=\f'_x(x)=0$, we must have $\p=\f'_x$.
  This amounts to $\f\le\f'_x$, and taking the supremum over
  all $\f$ gives $\f_x\le\f'_x$. On the other hand, $\f'_x$ is a competitor in the definition
  of $\f_x$, so $\f'_x\le\f_x$, and we conclude that $\f'_x=\f_x$, as claimed.

  It remains to prove that $\f_x$ is continuous. 
  Since $\f_x$ is psh, $\e_j:=\f_x(x_j)$ decreases to $\f_x(x)=0$.
  Now $\f_{x_j}\le\f_x\le\f_{x_j}+\e_j$ on $X$,
  so we see that $\f_{x_j}$ converges uniformly to $\f_x$.
  Since $\f_{x_j}$ is continuous for all $j$, $\f_x$ is continuous.

  \smallskip
  Finally suppose $L$ is a line bundle. We must prove 
  $\f_x=\FS(\|\cdot\|_{x,\bullet})=:\p_x$.
  In the definition of $\f_x$, it suffices to take the supremum over positive metrics. 
  Recall that $\p_x=\usc\sup_m\f_m$, where 
  $\f_m=m^{-1}\max\{\log\frac{|s|_0}{|s(x)|}\ s\in R_m\setminus\{0\}\}$.
  Since $\f_m(x)=0$, this gives $\f_m\le\f_x$ and hence $\p_x\le\f_x$
  since $\f_x$ is continuous. On the other hand, if $\f\in\cH(L)$ and $\f(x)\le 0$,
  then $\f\le\f_m$ for some $m$, and hence $\f\le\p_x$. Thus $\f_x=\p_x$.
  The equality $E(\f_x)=\vol(\|\cdot\|_{x,\bullet})$ now follows from Theorem~\ref{T407}
  (with $\|\cdot\|'_\bullet$ the trivial graded norm).
\end{proof}
\begin{rmk}
  Suppose $(\cX,\cL)$ is a normal test configuration for $(X,L)$ with irreducible central
  fiber $\cX_0$. This includes the case of ``special'' test configurations of~\cite{LX14}.
  Then the associated metric $\f=\f_{\cX,\cL}$ on $L$ satisfies
  $\MA(\f)=\delta_x$, where $x\in\Xdiv$ is the divisorial point corresponding to $\cX_0$.
  Therefore $\f_x=\f+c$ is a positive metric in this case. 
  More generally, if $x$ is a ``dreamy'' divisorial valuation in the sense of~\cite{FujitaValcrit},
  then $\f_x\in\cH(L)$. 
  However, given~\cite{Kur03}, we do not expect $\f_x\in\cH(L)$
  for a general divisorial point $x\in\Xdiv$.
\end{rmk}
Note that $x\mapsto\f_x$ is equivariant for the $\R_+^\times$-actions on $X$ and $\PSH(L)$:
we have $\f_{x^t}(y^t)=t\f_x(y)$ for $x,y\in X$, $t\in\R_+^\times$.
%
%
%
%
\subsection{Thresholds}
Recall (see \eg~\cite{CS08}) that the \emph{log canonical threshold} $\a(L)$
of $L$ is the infimum of the log canonical threshold $\lct(\cD)$, where
$\cD$ ranges over effective $\Q$-divisors $\Q$-linearly equivalent to $L$. 
We have $\a(L)>0$, see for instance~\cite[Theorem~9.14]{BHJ1}. In~\cite[Theorem~E]{BlumJonsson}
the following valuative formula was given.
\begin{thm}\label{T406}
  We have $\a(L)=\inf_x\frac{A(x)}{T(x)}$, 
  where $x$ may range over either divisorial points or
  nonpluripolar points. 
  Here $A$ is the log discrepancy and $T$ the Alexander-Taylor capacity.
\end{thm}
Indeed, this follows from~\cite[Theorem~E]{BlumJonsson} since $T(x)$ is equal
to the maximum order of vanishing of $x$.

In~\S\ref{sec:adjoint} we defined the adjoint stability threshold $\d(L)$
using measures of finite energy, see~\eqref{e411}. 
We now show that this definition coincides with the 
one in~\cite{BlumJonsson}.
\begin{thm}\label{T409}
  We have $\d(L)=\inf_x\frac{A(x)}{S(x)}$,
  where $x$ may range over either divisorial points or
  nonpluripolar points. 
  Here $A$ is the log discrepancy and $S$ the Monge--Amp\`ere energy.
\end{thm}
\begin{proof}
  We have seen that a point $x\in\Xval$ is nonpluripolar iff $S(x)<\infty$
  iff $T(x)<\infty$ iff the valuation $x$ has linear growth in the sense of~\cite{BKMS}.
  Further, for such $x$, the Monge--Amp\`ere energy $S(x)=E^*(\d_x)$ coincides
  with the expected order of vanishing as in~\cite{BKMS}.
  It therefore follows from~\cite[Theorem~E]{BlumJonsson}
  that the infima of $A(x)/S(x)$ over the set of divisorial points and the set of
  nonpluripolar points agree.
  Denote the common infimum by $\d'(L)$. It remains to show that $\d'(L)=\d(L)$.

  By considering Dirac masses $\mu=\d_x$, we get $\d(L)\le\d'(L)$. By definition of $\d'(L)$, we have $A(x)\ge\d'(L) S(x)$ for all divisorial points $x$, and hence for all $x\in X$ by (iv) in Theorem~\ref{T403} and Lemma~\ref{L409}. 
  On the other hand, for any measure $\mu\in\cM^1(X)$, we have 
  $\Ent(\mu)=\int_X A(x)\,d\mu(x)=\int_X\Ent(\d_x)\,d\mu(x)$
  and $E^*(\mu)\le\int_X E^*(\d_x)\,d\mu(x)$, see Proposition~\ref{P411}. This implies $\d'(L)\le\d(L)$ and completes the proof.
\end{proof}

The Alexander-Taylor inequality~\eqref{e409} together with Theorems~\ref{T406}
and~\eqref{T409}  immediately shows that 
\begin{equation}\label{e408}
  \frac{n+1}{n}\a(L)\le\d(L)\le(n+1)\a(L);
\end{equation}
in particular, $\d(L)>0$.
%
%
%
%
\subsection{A valuative criterion of adjoint K-stability}
In this section we prove Theorems~E and~E'. Recall that these say that to
test K-stability, it suffices to consider metrics in $\cE^1(L)$ of the form 
$\f=\f_x$, where $x$ is a nonpluripolar (or even divisorial) point.

First we note the following consequence of Theorem~\ref{T410}.
\begin{cor}\label{C401}
  For any nonpluripolar point $x\in X$, 
  \begin{equation*}
    H(\f_x)=A(x)
    \quad\text{and}\quad
    \Mad(\f_x)=A(x)-S(x).
  \end{equation*}
\end{cor}
Indeed, $\MA(\f_x)=\d_x$ implies 
$H(\f_x)=\Ent(\d_x)=A(x)$, $(I-J)(x)=E^*(\d_x)=S(x)$,
and hence $\Mad(\f_x)=A(x)-S(x)$.

To prove Theorem~E we use Corollary~\ref{C407}, which says that $L$ is 
K-semistable in the adjoint sense iff $\d(L)\ge 1$. 
By the valuative formula for $\d(L)$ in Theorem~\ref{T409}, this is equivalent to
$A(x)\le S(x)$ for all divisorial points $x$, or, equivalently, for all
nonpluripolar points $x$. Since $\Mad(\f_x)=S(x)-A(x)$, this proves Theorem~E.

The proof of Theorem~E' is similar. By Theorem~\ref{T405}, $L$ is uniformly
K-stable in the adjoint sense iff $\d(L)>1$, and this is equivalent to the existence
to $\e>0$ such that $(1+\e)S(x)\le A(x)$ for all divisorial (resp.\ nonpluripolar)
points $x\in X$. This is in turn equivalent to $\Mad(\f_x)\ge\e(I-J)(\f_x)$, 
and the proof is complete since the functionals $I-J$ and $J$ are comparable,
see~\eqref{e410}.
\begin{rmk}
  For the convenience of the reader, we compare our notation with the one in 
  Fujita~\cite{FujitaValcrit,FujitaPlt}. Assume $X$ is Fano, and $L=-K_X$.
  Let $F$ be a prime divisor over $X$, and write $x=\exp(-\ord_F)\in\Xdiv$.
  Then 
  $\tau(F)=T(x)$,
  $j(F)=V(T(x)-S(x))$,
  $\beta(F)=V(A(x)-S(x))$,
  and $\hat\beta(F)=1-S(x)/A(x)$.
\end{rmk}
%
%
%
%
\subsection{Adjoint K-stability and uniform K-stability}
Suppose that $k=\C$ (or $k$ algebraically closed and uncountable).
By~\cite[Theorem~E]{BlumJonsson}, there exists a nonpluripolar
point $x\in\Xval$ such that $A(x)=\d(L)S(x)>0$. 
This implies Theorem~F in the introduction. Indeed, if $L$
is K-stable in the adjoint sense, it is K-semistable in the adjoint
sense, so $\d(L)\ge1$ by Corollary~\ref{C407}. In view of Theorem~\ref{T405}
we must show that $\d(L)>1$. But if $\d(L)=1$, pick $x\in\Xval$ as above.
Then $\f_x$ is a nonconstant metric
in $\cE^1(L)$, and $\Mad(\f_x)=0$, contradicting $L$ being K-stable in
the adjoint sense.

If $L$ is not K-stable in the adjoint sense, \ie $\d(L)<1$, and $x\in\Xval$ is a
nonpluripolar point with $A(x)=\d(L)S(x)>0$, then $\f_x\in\cE^1(L)$
can be viewed as a ``maximally destabilizing metric''. 
In general, we don't know whether $\f_x\in\cH(L)$.
%
%
%
%
\subsection{The Tian--Odaka--Sano criterion}
In~\cite{Tian97}, Tian gave an analytic sufficient condition for the existence of
K\"ahler--Einstein metrics on Fano manifolds. 
An algebraic version of this was proved by Odaka and Sano~\cite{OS12};
see also~\cite{Der16, BHJ1}. Our next result gives a generalization to the general adjoint setting.
\begin{cor}
  Let $L$ be an ample $\Q$-line bundle. 
  If $\a(L)\ge\frac{n}{n+1}$ (resp.\ $\a(L)>\frac{n}{n+1}$),
  then $L$ is K-semistable (resp.\ uniformly K-stable) in the adjoint sense.
\end{cor}
\begin{proof}
  Immediate consequence of~\eqref{e408}, 
  Corollary~\ref{C407} and Theorem~\ref{T405}.
\end{proof}
The same proof also gives a necessary condition for adjoint K-stability.
\begin{cor}
  Let $L$ be an ample $\Q$-line bundle. 
  If $\a(L)<\frac{1}{n+1}$ (resp.\ $\a(L)\le\frac1{n+1}$),
  then $L$ cannot be K-semistable (resp.\ uniformly K-stable) in the adjoint sense.
\end{cor}
In the case of K-semistability, this generalizes~\cite[Theorem~3.5]{FO16}.
%
%
%
%
\subsection{A criterion of finite energy}
Using the positivity of the stability threshold, we now give a criterion for a measure to 
have finite energy.
\begin{cor}
  Any Radon probability measure of finite entropy has finite energy.
\end{cor}
\begin{proof}
  Indeed, if $\Ent(\mu)<\infty$, then $E^*(\mu)\le\d(L)^{-1}\Ent(\mu)<\infty$.
\end{proof}
\begin{cor}\label{C404}
  If $\mu$ is a Radon probability measure on $X$ such that the log discrepancy $A$ is 
  bounded above on the support of $\mu$, then $E^*(\mu)<\infty$.
\end{cor}
Applying this result to a Dirac mass $\mu=\d_x$ and using Theorem~\ref{T410}, we get
\begin{cor}
  If $x\in\Xval$ and $A(x)<\infty$, then $x$ is nonpluripolar.
\end{cor}
Remark~2.19 in~\cite{BKMS} contains an example of a point $x\in\Xval$
whose associated graded norm is not bounded; hence $x$ is pluripolar.
In this example, $A(x)=\infty$. On the other hand, it is also possible that 
$x$ is nonpluripolar even though $A(x)=\infty$. 
\begin{exam}\label{Ex401}
  Suppose $\dim X=2$ and that $\xi=\redu(x)$ is a closed point of $\Xsch$. 
  Thus $v:=-\log|\cdot|_x$ is a valuation (in the additive sense) on the local 
  ring $\cO_{\Xsch,\xi}$ centered at $\xi$, and after scaling 
  we may assume $\min_{f\in\fm}v(f)=1$, where $\fm$ is the maximal ideal.
  This means $v$ can be viewed as a point in the \emph{valuative tree} $\cV=\cV_\xi$
  at $\xi$, see~\cite{valtree} and also~\cite{dynberko}. The log
  discrepancy $A(x)$ equals the thinness $A(v)$ considered in~\cite{valtree}.
  By~\cite[Remark~A.4]{valtree} we can find $v$ of infinite thinness, 
  $A(v)=\infty$ but finite skewness,
  $\a(v)<\infty$. The latter condition means that there exists $C>0$ such that 
  $v\le C\ord_\xi$ on $\cO_{\Xsch,\xi}$, and implies that the associated graded norm 
  is bounded, see~\cite[Theorem~2.16]{BKMS}, and hence $T(x)<\infty$.
\end{exam}
\begin{cor}
  If $\mu$ is a Radon probability measure on $X$ whose support is contained in the dual cone
  complex $\D(\cY,\cD)$ of some log smooth pair $(\cY,\cD)$ over $X$, then $E^*(\mu)<\infty$.
\end{cor}
\begin{proof}
  Indeed, the support of $\mu$ must then be a compact subset of $\D(\cY,\cD)$, so since
  the log discrepancy is continuous on $\D(\cY,\cD)$, it must be bounded on the support.
\end{proof}
%
%
%
%
%
%
\appendix
\section{Properties of the log discrepancy}\label{a:logdisc}
%
%
%
%
In this section we prove Theorem~\ref{T403}.
We first \emph{define} the log discrepancy function $A_X$. 
On $X\setminus\Xval$, we declare $A_X\equiv\infty$, and 
on $\Xval$ we define $A_X$ as in~\cite[\S5]{jonmus}.
Let us recall how this is done.
Consider a log smooth pair $(\cY,\cD)$ over $\Xsch$.
The function $A_X$ is defined on the dual cone complex $\D(\cY,\cD)$
by~\eqref{e407} on the one-dimensional cones, and by linearity on the 
higher-dimensional cones. It is shown in~\cite{jonmus} that this 
gives a well-defined function on $\Xqm$. To extend $A_X$ to $\Xval$, we use 
the fact that the set of log smooth pairs is directed, and that we have a homeomorphism 
\begin{equation*}
  \Xval\simto\varprojlim_{(\cY,\cD)}\D(\cY,\cD),
\end{equation*}
where $\retr_{\cY,\cD}\colon\Xval\to\D(\cY,\cD)$ is a natural retraction defined by 
evaluation.
One can then show that the net $A_X\circ\retr_{\cY,\cD}$ is increasing and set
\begin{equation*}
  A_X(x)
  =\lim_{\cY,\cD}A_X(\retr_{\cY,\cD}(x))
  =\sup_{\cY,\cD}A_X(\retr_{\cY,\cD}(x))
\end{equation*}
for $x\in\Xval$.
Finally we set $A_X=\infty$ on $X\setminus\Xval$.

We claim that $A_X$ is lsc on $X$. Since $A_X\circ\retr_{\cY,\cD}$ is continuous 
for every $(\cY,\cD)$, the restriction of $A_X$ to $\Xval$
is lsc. Now $A_X=\infty$ on $X\setminus\Xval$, so this implies that 
$A_X$ is lsc at every point in $\Xval$.
It remains to prove that $A_X$ is lsc at any point $x\in X\setminus\Xval$.
There exists an ample line bundle $L$ on $X$ and a nonzero global section 
$s$ of $L$ such that $|s(x)|=0$. Now there exists $C=C(L)>0$ such that 
$\ord_\xi(s)\le C$ for all $\xi\in\Xsch$, cf~\cite[p.830]{BHJ1}.
By the Izumi inequality, see \eg~\cite[Proposition~5.10]{jonmus}, this gives
\begin{equation}\label{e202}
  |s(x')|
  \ge\exp(-A_X(x')\ord_\xi(x'))
  \ge\exp(-CA_X(x'))
\end{equation}
for every $x'\in\Xval$, where $\xi=\redu x'\in\Xsch$.
The same inequality trivially holds also when $x'\in X\setminus\Xval$,
since $A_X(x')=\infty$ in this case

Now, given any $B>0$, we have $|s|<\exp(-CB)$ on 
any sufficiently small neighborhood of $x$ in $X$, and hence
$A_X\ge B$ on the same neighborhood. 
This proves that $\lim_{x'\to x}A_X(x')=A_X(x)=\infty$.
In particular, $A_X$ is lsc at $x$, and hence lsc everywhere on $X$.

\smallskip
Next we prove that $A_X$ is the largest lsc extension of $A_X\colon\Xdiv\to\R$.
(This of course implies the uniqueness statement in Theorem~\ref{T403}.)
Let $A':X\to[0,\infty]$ be any other lsc extension. Since $A_X$ is continuous on
any dual cone complex $\D=\D(\cY,\cD)$, and $\Xdiv$ is dense in $\D$,
we have $A_X\ge A'$ on $\D$; hence $A_X\ge A'$ on $\Xqm$. 
Now suppose $x\in X\setminus\Xqm$. By construction, there exists a 
a net $(x_j)_j$ in $\Xqm$ such that $\lim A_X(x_j)=A_X(x)$.
Then $A'(x_j)\le A_X(x_j)$ for all $j$, so $A'(x)\le A_X(x)$ since $A'$
is lsc.

\smallskip
It remains to prove that $A$ satisfies~(iii) and~(iv) in Theorem~\ref{T403}.
For this, we use the Gauss extension $\sigma\colon X\to X\times\P^1$.
Its image consists of all $\bbG_m$-invariant points satisfying $\log|\unipar|=-1$,
and $\sigma(\Xval)\subset(X\times\P^1)^{\mathrm{val}}$. If $x\in\Xval$,
then $\sigma(x)$ is divisorial iff $x$ is either divisorial or the generic point of $x$.

Consider $\cX\in\SNC(X)$. We view $(\cX,\cX_0)$ as a log smooth pair
over $\Xsch\times\bbP^1$. The dual cone complex $\D(\cX,\cX_0)$ embeds in 
$(X\times\P^1)^{\mathrm{val}}\subset X\times\P^1$, $\sigma$ 
maps $\D_\cX$ homeomorphically onto the subset of $\D(\cX,\cX_0)$ cut out by 
the equation $\log|\unipar|=-1$, and $\sigma^{-1}(\D(\cX,\cX_0))=\D_\cX$.
The embedding $\sigma\colon\D_\cX\hookrightarrow\D(\cX,\cX_0)$ is affine 
in the sense that if $f\in C^0(\D(\cX,\cX_0))$ is linear on each cone, then 
$f\circ\sigma$ is affine on each simplex of $\D_\cX$. 
The retractions $\retr_\cX\colon X\to\D_\cX$ 
and $\retr_{\cX,\cX_0}\colon(X\times\P^1)^{\mathrm{val}}\to\D(\cX,\cX_0)$ satisfy
\begin{equation}\label{e428}
  \retr_{\cX,\cX_0}\circ\sigma=\sigma\circ\retr_\cX
  \quad\text{on $\Xval$}.
\end{equation}
\begin{lem}\label{L406}
  For any $x\in X$, we have $A_{X\times\P^1}(\sigma(x))=A_X(x)+1$.
\end{lem}
Granted this lemma, property~(iii) follows from~\eqref{e428}.
Indeed, $A_{X\times\P^1}$ is affine on each cone of 
$\D(\cX,\cX_0)$, so since $\sigma$ is affine, $A_X=A_{X\times\P^1}\circ\sigma-1$
is affine on each simplex of $\D_\cX$, proving~(a).
Similarly, by~\cite{jonmus} we have $A_{X\times\P^1}(y)\ge A_{X\times\P^1}(\retr_{\cX,\cX_0}(y))$
for $y\in(X\times\P^1)^{\mathrm{val}}$, with equality iff $y\in\D(\cX,\cX_0)$.
Using Lemma~\ref{L406}, this translates into~(iii)~(b).
Property~(iv) now follows formally. Indeed, if $x\in X$, then $x=\lim_\cX\retr_\cX(x)$, so 
$A_X(x)\le\varliminf_\cX A_X(\retr_\cX(x))$ since $A_X$ is lsc.
But $A_X(x)\le A_X(\retr_\cX(x))$ for all $\cX$, so~(iv) follows.
\begin{proof}[Proof of Lemma~\ref{L406}]
  For $x$ divisorial or the generic point of $X$, the equality is a special case 
  of~\cite[Proposition~4.11]{BHJ1}.
  Set $A':=A_{X\times\P^1}\circ\sigma-1$. Then $A'$ is lsc and $A'=A_X$ on $\Xdiv$,
  so $A'\le A_X$ by what precedes.

  Now consider $\cX\in\SNC(X)$, let $S$ be a simplex in $\D_\cX$, and $\hat{S}$
  the corresponding cone in $\D(\cX,\cX_0)$. 
  Then $A_{X\times\P^1}$ is continuous on $\hat{S}$, so $A'$ is continuous
  on $S$. Further, $A'=A_X$ on a dense subset of $S$, so $A_X\le A$ on $S$, since
  $A_X$ is lsc. Thus $A'=A_X$ on $\Xqm$. 

  Finally consider an arbitrary point $x\in X$. 
  If $x$ has nontrivial kernel, so does $\sigma(x)$ and 
  $A'(x)=A_X(x)=\infty$.  Now suppose $x\in\Xval$. The net 
  $(\retr_\cX(x))_\cX$ converges to $x$, so $\retr_{\cX,\cX_0}(x)=\sigma(\retr_\cX(x))$ 
  converges to $\sigma(x)$, by the continuity of $\sigma$. Now
  $A_{X\times\P^1}(\retr_{\cX,\cX_0}(\sigma(x)))\le A_{X\times\P^1}(\sigma(x))$ for
  every $x$, so $\lim_\cX A_{X\times\P^1}(\retr_{\cX,\cX_0}(\sigma(x))=A_{X\times\P^1}(\sigma(x))$
  by lower semicontinuity. 
  Thus $A'(x)=\lim_\cX A'(\retr_\cX(x))=\lim_\cX A_X(\retr_\cX(x))\ge A_X(x)$,
  where the last inequality follows since $A_X$ is lsc. 
  Since $A'\le A_X$, this completes the proof.
\end{proof}
%
%
%
%


\end{document}